\documentclass[10pt, reqno]{amsart}

\usepackage[T1]{fontenc}
\usepackage[english]{babel}
\usepackage[dvipsnames,table,xcdraw]{xcolor}
\definecolor{myred}{RGB}{228,26,28}
\definecolor{myblue}{RGB}{55,124,184}
\definecolor{mygreen}{RGB}{77,175,74}
\usepackage[colorlinks=true,urlcolor=myblue,linkcolor=myblue,runcolor=myred,citecolor=myred,linktoc=all]{hyperref}

\usepackage{geometry}
\usepackage{graphicx}
\usepackage{array}
\usepackage{wrapfig}
\usepackage{enumitem}
\setlist[itemize]{label={$\vcenter{\hbox{\tiny$\bullet$}}$}}
\usepackage{appendix}
\usepackage{caption}
\usepackage{subcaption}

\usepackage{amsmath,mathtools,amssymb,amsthm,amsfonts}
\usepackage[foot]{amsaddr}
\usepackage{bm}
\usepackage{esint}
\usepackage{chemformula}

\renewcommand{\lq}{\leqslant}
\newcommand{\gq}{\geqslant}
\newcommand{\set}[1]{\left\{ #1 \right\}}

\newcommand{\Rb}{\mathbb{R}}
\newcommand{\Cb}{\mathbb{C}}

\newcommand{\Forall}{\forall\ }
\newcommand{\ie}{\emph{i.e.}\ }
\newcommand{\Ran}{{\rm Ran} \, }
\newcommand{\Span}{{\rm Span} \, }
\renewcommand{\Re}{{\rm Re} \, }

\newcommand{\Tr}{{\rm Tr}}
\renewcommand{\d}{{\rm d}}
\newcommand{\prt}[1]{\left( #1 \right)}

\renewcommand{\i}{{\rm i}}
\renewcommand{\L}{{\rm L}}
\renewcommand{\H}{{\rm H}}

\newcommand{\herm}{{\rm herm}}

\newcommand{\normF}[1]{\left\lVert#1\right\rVert_{\rm F}}
\newcommand{\norm}[1]{\left\| #1 \right\|}
\newcommand{\abs}[1]{\left| #1 \right|}
\newcommand{\cro}[1]{\left\langle  #1 \right\rangle}
\newcommand{\croF}[1]{\left\langle  #1 \right\rangle_{\rm F}}
\newcommand{\bra}[1]{\left\langle #1 \right|}
\newcommand{\ket}[1]{\left| #1 \right\rangle}
\newcommand{\bak}[3]{\left\langle #1 \middle| #2 \middle | #3 \right\rangle}

\newcommand{\Hc}{\mathcal{H}}
\newcommand{\Nc}{\mathcal{N}}
\newcommand{\Sc}{\mathcal{S}}
\renewcommand{\Mc}{\mathcal{M}}
\newcommand{\Tc}{\mathcal{T}}
\newcommand{\Lc}{\mathcal{L}}
\newcommand{\Xc}{\mathcal{X}}
\newcommand{\Gc}{\mathcal{G}}
\newcommand{\Rc}{\mathcal{R}}

\newcommand{\Op}{{\Omega}}
\def\deltaP{X}
\def\deltaPhi{\Xi}
\def\deltaphi{\xi}

\newcommand{\Enl}{E_{\rm nl}}
\newcommand{\Ecut}{{E_{\rm cut}}}
\newcommand{\Ecutref}{{E_{\rm cut,ref}}}
\newcommand{\Nel}{{N_{\rm el}}}

\theoremstyle{plain}
\newtheorem{proposition}{Proposition}

\theoremstyle{definition}
\newtheorem{remark}{Remark}

\addto\extrasenglish{%
}

\title{Practical error bounds for properties in plane-wave electronic
  structure calculations}
\author{Eric Canc\`es$^{1,2}$}
\author{Genevi\`eve Dusson$^{3}$}
\author{Gaspard Kemlin$^{1,2}$}
\author{Antoine Levitt$^{1,2}$}

\address[1]{CERMICS, \'Ecole des ponts}
\address[2]{Inria Paris}
\address[3]{Laboratoire de Math\'ematiques de Besan\c{c}on, UMR CNRS 6623,
  Universit\'e Bourgogne Franche-Comt\'e}
\email{eric.cances@enpc.fr, genevieve.dusson@math.cnrs.fr,
  gaspard.kemlin@enpc.fr, antoine.levitt@inria.fr}

\begin{document}

\begin{abstract}
  We propose accurate computable error bounds for quantities of
  interest in plane-wave electronic structure calculations, in particular
  ground-state density matrices and energies,
  and interatomic forces. These bounds are based on an
  estimation
  of the error in terms of the residual of the solved equations, which
  is then efficiently approximated with computable terms. After
  providing coarse bounds based on an analysis of the inverse Jacobian, we
  improve on these bounds by solving a linear problem in a small
  dimension that involves a Schur complement. We numerically show how
  accurate these bounds are on a few representative materials, namely
  silicon, gallium arsenide and titanium dioxide.
\end{abstract}

\maketitle

\section{Introduction}

This article focuses on providing practical error estimates for
numerical approximations of electronic structure calculations. Such computations are key in many domains, as they allow for the simulation of systems
at the atomic and molecular scale. They are particularly
useful in the fields of chemistry, materials science, condensed matter
physics, or molecular biology. Among the many electronic structure
models available, Kohn--Sham (KS) Density Functional Theory
(DFT)~\cite{Kohn1965} with semilocal
density functionals is one of the most used in practice, as it offers
a good compromise between accuracy and computational cost. We will
focus on this model in this article. Note that the mathematical
formulation of this problem is similar to that of the
Hartree--Fock~\cite{hartreeWaveMechanicsAtom1928} or Gross--Pitaevskii
equations~\cite{pitaevskiiBoseEinsteinCondensation2003}, so that what
we present in the context of DFT can be easily extended to such
contexts. We will focus on plane-wave discretizations within the
pseudopotential approximation, which are most
suited for the study of crystals; some (but not all) of our methodology can be applied in
other contexts as well, including the aforementioned Hartree--Fock or Gross--Pitaevskii models, as well as molecules simulated using plane-wave DFT.

In this field, the first and most crucial problem is the determination of the
electronic ground state of the system under consideration. Mathematically
speaking, this problem is a constrained minimization problem. Writing
the first-order optimality conditions of this problem leads to an
eigenvalue problem that is nonlinear in the eigenvectors. At the continuous level, the unknown is a
subspace of dimension $\Nel$, the number of electrons in the system;
this subspace can be conveniently described using either the orthogonal projector on it (density
matrix formalism) or an orthonormal basis of it (orbital formalism). This problem is
well-known in the literature and the interested reader is referred to
\cite{cancesConvergenceAnalysisDirect2021} and the references therein for
more information on how it is solved in practice.

Solving this problem numerically requires a number of
approximations, so that only an approximation of the exact
solution can be computed. Being able to estimate the error between the
exact and the approximate solutions is crucial, as this information can
be used to reduce the high computational cost of such methods by an
optimization of the approximation parameters, and maybe more
importantly, to add error bars to quantities of interest (QoI)
calculated from the approximate solution. In our context, such QoI are
typically the ground-state energy of the system and the forces on the atoms
in the system, but may also include any information computable from
the Kohn--Sham ground state.

While such error bounds have been developed already some time ago for
boundary value problems, e.g. in the context of finite element
discretization, using in particular {\it a posteriori} error
estimation~\cite{verfurthPosterioriErrorEstimation1994}, the
development of bounds in the context of electronic structure is
quite recent, and still incomplete. Computable and guaranteed error
bounds for linear eigenvalue problems have been proposed in the last
decade~\cite{cancesGuaranteedRobustPosteriori2017,cancesGuaranteedRobustPosteriori2018,cancesGuaranteedPosterioriBounds2020,carstensenGuaranteedLowerBounds2014,liuFrameworkVerifiedEigenvalue2015a};
we refer to~\cite[Chapter 10]{nakaoNumericalVerificationMethods2019a}
for a recent monograph on the subject. Specifically for electronic
structure calculations, some of us proposed guaranteed error bounds
for linear eigenvalue
equations~\cite{herbstPosterioriErrorEstimation2020a}. For fully
self-consistent (nonlinear) eigenvalue equations, an error bound was
proposed for a simplified 1D Gross--Pitaevskii equation
in~\cite{dussonPosterioriAnalysisNonlinear2017}; however the
computational cost of evaluating the error bound in this contribution
is quite high. So far, no error bound has been proposed for the error estimation of
general QoI in electronic structure calculations, in particular for
the interatomic forces. This is what we are trying to achieve in
this contribution.

In this article, we use a general approach based on a
linearization of the Kohn--Sham equations. It is instructive to compare our
approach to those used in a general context. Assume we want to find
$x \in \Rb^{n}$ such that $f(x)=0$, for some nonlinear function
$f : \Rb^{n} \to \Rb^{n}$ (the residual). Near a solution $x_{*}$, we have
$f(x) \approx f'(x) (x-x_{*})$, and therefore, if $f'(x)$ is
invertible, we have the error-residual relationship
\begin{align}
  \label{eq:linearization}
  x-x_{*} \approx f'(x)^{-1} f(x).
\end{align}
This is the same approximation that leads to the Newton algorithm.
Assume now that we want to compute a real-valued
QoI $A(x_*)$, where $A:\Rb^n \to \Rb$ is a $C^1$ function (e.g. the energy, a component of the interatomic forces, of the density, \dots); then  we have the approximate equality with computable right-hand side:
\begin{align}\label{eq:err_A}
  A(x)-A(x_{*}) \approx \nabla A(x) \cdot \prt{ f'(x)^{-1} f(x)}.
\end{align}
From here, we obtain the simple first estimate
\[
  |A(x) - A(x_*)| \lq  \abs{\nabla A(x)} \norm{f'(x)^{-1}}_{\rm op}
  \abs{f(x)},
\]
where $\abs{\cdot}$ is any chosen norm on $\Rb^n$, and
$\norm{\cdot}_{\rm op}$ is the induced operator norms on
$\Rb^{n \times n}$ (note that $\nabla A(x) \in \Rb^{n}$ and
$f'(x) \in \Rb^{n \times n}$). This approximate bound can be turned
into a rigorous one using information on the second derivatives of $f$; see for instance~\cite{schmidtRigorousEffectiveAposteriori2020}.

In extending this approach to Kohn--Sham models, we
encounter several difficulties:
\begin{itemize}
  \item The structure of our problem is not easily formulated as above
    because of the presence of constraints and degeneracies. We solve this using the
    geometrical framework of~\cite{cancesConvergenceAnalysisDirect2021}
    to identify the appropriate analog to the Jacobian $f'(x)$.
  \item The computation of the Jacobian and its inverse is prohibitively
    costly. We use iterative strategies to keep this cost manageable.
  \item Choosing the right norm is not obvious in this context. For
    problems involving partial differential equations, where $f$
    includes partial derivatives, it is natural to consider Sobolev-type
    norms, with the aim of making $f'$ a bounded operator between the
    relevant function spaces. We explore different choices and their
    impacts on the error bounds.
  \item The operator norm inequalities
    \[
      \left| \nabla A(x) \cdot \prt{ f'(x)^{-1} f(x)} \right| \lq
      \abs{\nabla A(x)}  \norm{f'(x)^{-1}}_{\rm op} \abs{f(x)}
    \]
    are very often largely suboptimal, even
    with appropriate norms. We quantify this on representative examples.
  \item The structure of the residual $f(x)$ plays an
    important role. For instance, when~$x$ results from a Galerkin
    approximation to a partial differential equation, $f(x)$ is
    orthogonal to the approximation space. In the context of plane-wave
    discretizations, this means the residual only contains
    high-frequency Fourier components. We
    demonstrate how this impacts the quality of the above
    bounds when $A$ represents the interatomic forces, in which case its derivative mostly acts on low-frequency components.
\end{itemize}

The main result of our work therefore lies in the derivation of an efficient, asymptotically accurate, way of approximating $\nabla A(x) \cdot \prt{ f'(x)^{-1} f(x)}$ using the specific structure of the residual $f(x)$ in a plane-wave discretization, where $A$ represents a component of the interatomic forces of the system. This approximation can then be used either to approach the actual error $A(x) - A(x_*)$ in \eqref{eq:err_A} or to improve $A(x)$ by computing $A(x) - \nabla A(x) \cdot \prt{ f'(x)^{-1} f(x)}$, which is a better approximation of $A(x_*)$. These estimates are a new step towards robust and guaranteed {\it a posteriori} error estimates for Kohn-Sham models: this paper reflects the process that lead to their derivation, by describing the issues raised when applying natural ideas and how we propose to solve these issues.

Throughout the paper, we will provide numerical tests to illustrate
our results. All these tests are performed with the DFTK software
\cite{herbstDFTKJulianApproach2021}, a recent Julia package solving
the Kohn--Sham equations in the pseudopotential approximation using a
plane-wave basis, thus particularly well suited for periodic systems
such as crystals~\cite{martinElectronicStructureBasic2004}. We are
mostly interested in three QoI: the ground-state energy, the ground-state density, and the
interatomic forces, the latter being computed using the
Hellmann--Feynman theorem. We will demonstrate the main points with
simulations on a simple system (bulk silicon), then present results
for more complicated systems.

We will be interested in this paper only in quantifying the
discretization error. However, the general framework we develop can be
used also to treat other types of error (such as the ones resulting
from the iterative solution of the underlying minimization problem). We only consider
insulating systems at zero temperature, and do not consider the error
due to finite Brillouin zone sampling
\cite{cancesNumericalQuadratureBrillouin2020}; extending the formalism
to finite temperature and quantifying the sampling error, particularly
for metals, is currently under investigation, see \cite{herbstRobustEfficientLine2022} for a first step in this direction.

The outline of this article is as follows. In \autoref{sec:sec2}, we
present the mathematical framework related to the solution of the
electronic structure minimization problem, describing in particular
objects related to the tangent space of the constraint manifold.
In \autoref{sec:kohn_sham}, we present the Kohn--Sham model and the numerical framework in which our tests will be performed.
In \autoref{sec:sec3}, we propose a first crude bound of the error
between the exact and the numerically computed solution based on a
linearization argument as well as an operator norm inequality, both
for the error on the ground-state density matrix and on the forces. In
\autoref{sec:sec4}, we refine this bound by splitting between low
and high frequencies, and using a Schur complement to refine the error bound on
the low frequencies. Finally, in \autoref{sec:experiments}, we provide
numerical simulations on more involved materials systems, namely on a
gallium arsenide system (\ch{GaAs}) and a titanium dioxide system
(\ch{TiO2}), showing that the proposed bounds work well in
those cases.

\section{Mathematical framework}
\label{sec:sec2}
In this section, we present the models targeted by our study, as well as the elementary tools of differential geometry used to derive and compute the error bounds on the QoI.

\subsection{General framework}

The work we present here is valid for a large class of mean-field
models including different instances of Kohn--Sham models, the
Hartree--Fock model, and the time-independent Gross--Pitaevskii model
and its various extensions. To study them in a unified way, we use a
mathematical framework similar to the one in
\cite{cancesConvergenceAnalysisDirect2021}. To keep the presentation
simple, we will work in finite dimension and consider that the
solutions of the problem can be very accurately approximated in a
given finite-dimensional space of (high) dimension $\Nc$, which we
identify with $\Cb^\Nc$. We denote by
\[\cro{x , y}\coloneqq \Re(x^*y)\] the $\ell^2$ inner product of
$\Cb^\Nc$, seen as a vector space over $\Rb$. We equip the $\Rb$-vector space of square Hermitian matrices
\[
  \Hc\coloneqq\Cb_{\herm}^{\Nc \times \Nc}
\]
with the Frobenius inner product $\croF{A , B}
\coloneqq\Re\prt{\Tr(A^*B)}$. Note that although it is important in
applications to allow for complex orbitals and density
matrices, the space of Hermitian matrices is not a vector space over
$\Cb$, and therefore we will always consider vector spaces over $\Rb$.

\medskip

The density-matrix formulation of the mean-field model in this reference approximation space reads
\begin{equation}
  \label{eq:pb}
  \min\{E(P), \; P\in\Mc_{\Nc} \}, \quad \mbox{where} \quad \Mc_{\Nc} \coloneqq \set{P\in\Hc \;\middle|\; P^2=P,\
    \Tr(P)=\Nel}
\end{equation}
is the manifold of rank-$\Nel$ orthogonal projectors (density matrices), and $E : \Hc \to \Rb$ is a $C^2$ nonlinear energy functional. The parameter $\Nel$ (with $1 \lq \Nel \lq \Nc$) is a fixed integer depending on the physical model, and not on its discretization in a finite-dimensional space. For mean-field electronic structure models, $\Nel$ is the number of electrons or electron pairs in the system (hence the notation $\Nel$); in the standard Gross--Pitaevskii model, $\Nel=1$. The energy functional $E$ is of the form
\[
  E(P) \coloneqq \Tr(H_0P) + \Enl(P),
\]
where $H_0$ is the linear part of the mean-field Hamiltonian, and
$\Enl$ a nonlinear contribution depending on the considered model (see
\autoref{sec:kohn_sham} for the expressions in the
Kohn--Sham model). For simplicity of presentation we will ignore spin in the formalism, but we will include it
in the numerical simulations (see Remark \ref{rem:spin}). The
set $\Mc_{\Nc}$ is diffeomorphic to the Grassmann manifold
of $\Nel$-dimensional complex vector subspaces of~$\Cb^{\Nc}$.

Problem \eqref{eq:pb} always has a minimizer since it consists in
minimizing a continuous function on a compact set. This minimizer may or may not
be unique, depending on the model and/or the physical system
under study. We will not elaborate here on this uniqueness issue, and
assume for the sake of simplicity that \eqref{eq:pb} has a unique
minimizer, which we denote by $P_*$.

Besides the ground-state energy
$E(P_{*})$, we can compute from $P_*$ various other physical quantities of
interest (QoI), for instance, the electronic density and the interatomic forces. We
denote such a QoI by $A_*=A(P_*)$.

\medskip

We consider the case when $\Nc$ is too large for problem \eqref{eq:pb} to be solved completely in the reference approximation space. To solve problem \eqref{eq:pb}, we therefore consider a finite-dimensional subspace $\Xc$ of $\Cb^\Nc$ of dimension $N_b$ and solve instead the variational approximation of \eqref{eq:pb} in $\Xc$, namely
\begin{equation}
  \label{eq:pb_Nb}
  \min\{E(P), \; P\in \Mc_{\Nc}, \; \Ran(P) \subset \Xc \}.
\end{equation}
Our goal is then to estimate the errors $\norm{A(P)-A_*}$
on the QoI $A$, where $P$
is typically the minimizer of \eqref{eq:pb_Nb}, given the variational space
$\Xc$, and the norm is specific to the QoI. The latter can be a
scalar, e.g. the ground-state energy, or a finite or infinite
dimensional vector, e.g. the interactomic forces, or the ground-state
density.

\subsection{First-order geometry}

The manifold $\Mc_{\Nc}$ is a smooth manifold. Its tangent space $\Tc_P\Mc_{\Nc}$ at $P\in\Mc_{\Nc}$ is given by
\begin{align*}
  \Tc_P\Mc_{\Nc} &= \set{\deltaP\in\Hc \;\middle|\; P\deltaP + \deltaP P = \deltaP,\ \Tr(\deltaP) = 0} \\ & =
  \set{\deltaP\in\Hc \;\middle|\; P\deltaP P = 0, P^\perp \deltaP P^\perp= 0},
\end{align*}
where $P^\perp=1-P$ is the orthogonal projection on
$\Ran(P)^\perp$ for the canonical inner product of $\Cb^\Nc$. The set $\Tc_P\Mc_{\Nc}$ is the set of Hermitian
matrices that are off-diagonal in the block decomposition induced by
$P$ and $P^{\perp}$; more explicitly, if $P=U\prt{ \begin{array}{cc} I_{\Nel} & 0 \\ 0 & 0 \end{array}}U^*$ for some unitary $U \in {\rm U}(\Nc)$, then
\[
  \Tc_P\Mc_{\Nc} =\set{X=U\prt{ \begin{array}{cc} 0 & Y^{*} \\ Y &
        0 \end{array}}U^*,
    \; Y
    \in
    \Cb^{(\Nc-\Nel)
      \times
      \Nel}}.
\]
The orthogonal projection ${\bm \Pi_{P}}$ on $\Tc_P\Mc_{\Nc}$ for the Frobenius inner product is given by
\begin{equation}
  \label{eq:prop_proj}
  \Forall \deltaP\in\Hc,\quad {\bm \Pi_{P}}(\deltaP) = P \deltaP P^\perp + P^\perp \deltaP P = [P,[P,\deltaP]] \in \Tc_P\Mc_{\Nc},
\end{equation}
where $[A,B] \coloneqq AB - BA$ is the commutator of $A$ and $B$. Linear operators acting on spaces of matrices are sometimes referred to as \emph{super-operators} in the physics literature. Throughout this paper, super-operators will be written in bold fonts.

The mean-field Hamiltonian is the gradient of the energy at a given
point $P$ (again for the Frobenius inner product):
\[
  H(P) \coloneqq \nabla E(P) = H_{0} + \nabla \Enl(P).
\]
To simplify the notation, we set
\begin{align} \label{eq:defH*}
  H_{*} = H(P_{*}) = \nabla E(P_*).
\end{align}
The first-order optimality condition for \eqref{eq:pb} is that $\nabla E(P_*)$ is orthogonal to the tangent space $\Tc_{P_*}\Mc_{\Nc}$, which can be written, using \eqref{eq:prop_proj} and \eqref{eq:defH*}, as ${\bm \Pi_{P_*}} H(P_*) =[P_*,[P_*,H(P_*)]]=0$. This corresponds to the residual
\begin{align*}
  R(P) =  {\bm \Pi_{P}} H(P) = [P,[P,H(P)]]
\end{align*}
being zero at $P_{*}$. The residual function $R$ can be seen as a nonlinear map from $\Hc$ to itself, and its restriction to $\Mc_\Nc$ as a vector field on $\Mc_\Nc$ since for all $P \in \Mc_\Nc$, $R(P) \in \Tc_P\Mc_\Nc$.

\subsection{Second-order geometry}

We introduce the super-operators ${\bm \Op}(P)$ and ${\bm K}(P)$, defined at $P\in\Mc_N$ and acting on $\Hc$. These
operators were already introduced in~\cite[Section 2.2]{cancesConvergenceAnalysisDirect2021}, but we recall here their
definitions for completeness. To simplify the notation, we will set $\bm{K}_* \coloneqq \bm{K}(P_*)$, $\bm{\Op}_* \coloneqq \bm{\Op}(P_*)$.

\medskip

The super-operator $\bm{K}_* \in \Lc(\Hc)$ is the Hessian of the energy projected onto the tangent space to $\Mc_\Nc$ at $P_*$:
\begin{align} \label{eq:defK*}
  \bm{K}_{*} \coloneqq {\bm \Pi_{P_*}}{{\bm \nabla}^2E}(P_{*}){\bm \Pi_{P_*}} = {\bm \Pi_{P_*}}{{\bm \nabla}^2\Enl}(P_{*}){\bm \Pi_{P_*}}.
\end{align}
By construction, $\Tc_{P_*}\Mc_\Nc$ is an invariant subspace of $\bm{K}_{*}$. Note that $\bm{K}_{*} =0$ for linear eigenvalue problems, \ie when $E_{\rm  nl}=0$.

The super-operator $\bm{\Op}_* \in \Lc(\Hc)$ is defined by
\begin{align} \label{eq:defOmega*}
 \Forall X \in \Hc, \quad  \bm{\Op}_*\deltaP = -[P_*,[H_{*} ,\deltaP]].
\end{align}
The tangent space $\Tc_{P_*}\Mc_\Nc$ is also an invariant subspace of $\bm{\Op}_*$.

It is shown in~\cite{cancesConvergenceAnalysisDirect2021} that the energy of a density matrix
$P=P_{*} +\deltaP + O\prt{\normF{\deltaP}^{2}} \in \Mc_{\Nc}$ with
$\deltaP \in \Tc_{P_*}\Mc_{\Nc}$ is $E(P)=E(P_*)+\croF{\deltaP,
  (\bm{\Op}_*+\bm{K}_*) \deltaP} + o\prt{\normF{\deltaP}^{2}} $.
The restriction of the operator $\bm{\Op}_* + \bm{K}_*$ to the invariant subspace $\Tc_{P_*}\Mc_\Nc$ can therefore be identified with the second-order derivative of $E$ on the manifold $\Mc_\Nc$. Since $P_{*}$ is a minimum, it follows that
\begin{align*}
  \bm{\Op}_* + \bm{K}_* \gq 0 \quad \mbox{on } \Tc_{P_*}\Mc_\Nc.
\end{align*}

Note that in general, the second-order derivative of a function
defined on a smooth manifold is not an intrinsic object; it depends
not only on the tangent structure of the manifold, but also on the
chosen affine connection. However, at the critical point $P_{*}$ of
$E$ on the manifold, the contributions to the second derivative due
the connection vanish and the second-order derivative becomes
intrinsic.

For our purposes, it will be convenient to define this second-order
derivative also outside of $P_{*}$. Relying on
\eqref{eq:defK*}-\eqref{eq:defOmega*}, we define for any
$P \in \Mc_{\Nc}$ the super-operators $\bm{\Op}(P) \in \Lc(\Hc)$ and
$\bm{K}(P) \in \Lc(\Hc)$ through
\begin{align} \label{eq:def_Omega_K}
  \Forall X \in \Hc, \quad  {\bm \Op}(P) X &= -[P,[H(P) ,X]] \quad \mbox{and} \quad
  \bm{K}(P) X = {\bm \Pi_{P}}{{\bm \nabla}^2E}(P){\bm \Pi_{P}} X.
\end{align}
Both $\bm{\Op}(P)$ and $\bm{K}(P)$ admit $\Tc_P\Mc_\Nc$ as an invariant
subspace and their restrictions to $\Tc_P\Mc_\Nc$ are Hermitian for the
Frobenius inner product. The map $\Mc_\Nc \ni P \mapsto
\bm{\Op}(P)+\bm{K}(P) \in \Lc(\Hc)$ is smooth and the restriction of
$\bm{\Op}(P)+\bm{K}(P)$ to $\Tc_P\Mc_\Nc$ provides a computable
approximation of the second-order derivative of $E: \Mc_\Nc \to \Rb$ in
a neighborhood of $P_*$ (whatever the choice of the affine connection).

\subsection{Density matrix and orbitals}
The framework we have outlined above is particularly convenient for
stating the second-order conditions, but much too expensive
computationally as it requires the storage and manipulation of
(low-rank) large matrices. In practice, it is more effective to work
directly with orbitals, \ie write for any $P \in \Mc_{\Nc}$

\begin{align}
  \label{eq:identification_manifold}
  P = \Phi \Phi^{*} = \sum_{i=1}^{\Nel} \ket{\phi_{i}}\bra{\phi_{i}}
\end{align}

\noindent
where $\Phi = (\phi_{1}|\cdots|\phi_{\Nel})$ is a collection of $\Nel$
orbitals $\phi_{i} \in \Cb^{\Nc}$ satisfying $\Phi^{*}\Phi=I_{\Nel}$ and $\Span(\phi_1,\dots,\phi_\Nel)=\Ran(P)$, and where we used Dirac's bra-ket notation: for $\phi,\psi \in \Cb^\Nc$, $\cro{\phi,\psi} =\phi^*\psi$ and $\ket{\phi}\bra{\psi}=\phi\psi^*$.
Problem \eqref{eq:pb} can be
reformulated as
\begin{align*}
  \min \set{E(\Phi \Phi^{*}), \Phi \in \Cb^{\Nc \times \Nel}, \Phi^{*} \Phi = I_{\Nel}}.
\end{align*}
Note that the orbitals are only defined up to a unitary transform: if
$U \in {\rm U}(\Nel)$ is a unitary matrix, then $\widetilde \Phi \coloneqq \Phi U$ and $\Phi$ give
rise to the same density matrix. This means that the minimizers of
this minimization problem are never isolated, which creates technical
difficulties that are not present in the density matrix formalism.

Let us fix a $\Phi=(\phi_1|\cdots|\phi_\Nel) \in \Cb^{\Nc \times \Nel}$ with $\Phi^{*} \Phi =
I_{\Nel}$, and consider an element $X$ of the tangent plane $\Tc_{\Phi \Phi^{*}}\Mc_{\Nc}$.
By completing $\Phi$ to an orthogonal basis and writing out $X$ in
this basis, it is easy to see that
the constraints $X^{*}=X$, $P X P=0$, $P^{\perp} X P^{\perp} =0$ imply that
$X$ can be put in the form

\begin{equation}
  \label{eq:identification_tangent_plane}
  X = \sum_{i=1}^{\Nel} \ket{\phi_{i}}\bra{\deltaphi_{i}} + \ket{\deltaphi_{i}}\bra{\phi_{i}} = \Phi \deltaPhi^{*} + \deltaPhi \Phi^{*}
\end{equation}

\noindent
where $\deltaPhi = (\deltaphi_{1}|\cdots|\deltaphi_{\Nel}) \in \Cb^{\Nc \times \Nel}$ is a
set of orbital variations satisfying $\Phi^{*}\deltaPhi=0$.
Furthermore, under this condition, $\deltaPhi$ is unique, so that
\eqref{eq:identification_tangent_plane} establishes a bijection
between $\Tc_{\Phi \Phi^{*}}\Mc_{\Nc}$ and $\{\deltaPhi \in \Cb^{\Nc \times
  \Nel} \; | \;  \Phi^{*}\deltaPhi = 0\}$. We will therefore treat equivalently elements
of the tangent space $\Tc_{P}\Mc_{\Nc}$ either in the density matrix
representation $X$ or the orbital representation $\deltaPhi$, writing
\begin{align}
  \deltaPhi \simeq_{\Phi} X.
\end{align}
This orbital representation of $P$ by $\Phi$ is more
economical computationally, only requiring the storage and
manipulation of orbitals $\Phi \in \Cb^{\Nc \times \Nel}$ satisfying
$\Phi^{*} \Phi=I_{\Nel}$. Similarly, the manipulation of objects $X$ in
the tangent plane $\Tc_{\Phi \Phi^{*}}\Mc_{\Nc}$ is more efficiently
done through their orbital variations $\deltaPhi \in \Cb^{\Nc \times \Nel}$
satisfying $\Phi^{*}\deltaPhi=0$.

All operations on density matrices or their variations can indeed be carried
out in this orbital representation. For instance, the computation of
the energy can be performed efficiently in practice, as explained in
\autoref{sec:kohn_sham}, and the residual at $P=\Phi\Phi^*$ also has a nice representation in terms of orbitals:
\begin{align*}
  R(\Phi \Phi^{*}) \simeq_{\Phi} H \Phi - \Phi (\Phi^{*}H\Phi) \quad \mbox{with $H$ evaluated at $\Phi \Phi^{*}$,}
\end{align*}
which is easily recognized as similar to the residual of a linear eigenvalue problem.

Likewise, operators on $\Tc_{\Phi \Phi^{*}}\Mc_{\Nc}$ can be identified
in this fashion. For instance,
\begin{equation}\label{eq:Omega_orbitals}
  {\bm{\Op}(\Phi\Phi^{*})} (\Phi \deltaPhi^{*} + \deltaPhi \Phi^{*}) \simeq_{\Phi} P^{\perp}(H \deltaPhi-\deltaPhi(\Phi^{*} H\Phi)) \quad \mbox{with $H$ evaluated at $\Phi \Phi^{*}$.}
\end{equation}
The computation of ${\bm{K}}$ can be performed similarly:
\begin{equation}\label{eq:K_orbitals}
  {\bm{K}(\Phi\Phi^{*})} (\Phi \deltaPhi^{*} + \deltaPhi \Phi^{*}) \simeq_{\Phi} P^{\perp}(\delta H \phi_i)_{i=1,\dots,\Nel} \quad \mbox{with } \delta H = \frac{\d H}{\d P}(\Phi \deltaPhi^{*} + \deltaPhi \Phi^{*}).
\end{equation}
Finally, note that all the numerical results in this article are performed using the orbital formalism.

\begin{remark}
  Note that the condition that $\Phi^{*}\deltaPhi=0$ is not necessary for
  $\Phi \deltaPhi^{*} + \deltaPhi \Phi^{*}$ to define an element of
  $\Tc_{\Phi \Phi^{*}}\Mc_{\Nc}$. However, without this gauge
  condition, $\deltaPhi$ is not unique. This is simply a manifestation
  at the infinitesimal level of the noninjectivity of the map
  $\Phi \mapsto \Phi\Phi^{*}$ between $\{\Phi \in \Cb^{\Nc \times \Nel},
  \Phi^{*} \Phi = I_{\Nel}\}$ and $\Mc_{\Nc}$. Because of this, the
  derivative $\deltaPhi \mapsto \Phi \deltaPhi^{*} + \deltaPhi \Phi^{*}$ is not
  injective between the tangent spaces $\{\deltaPhi \in \Cb^{\Nc \times \Nel},
  \Phi^{*} \deltaPhi + \deltaPhi^{*} \Phi = 0\}$ and $\Tc_{\Phi \Phi^{*}}\Mc_{\Nc}$.
  In more concrete terms, in the example case where $\phi_{1},
  \dots, \phi_{N_{\rm el}}$ are the first $N_{\rm el}$ basis vectors, any
  element $X$ is of the form
  $  \begin{pmatrix}
    0&Z^*\\
    Z&0
  \end{pmatrix}
  $ which can be written in the form
  \eqref{eq:identification_tangent_plane} with $\deltaPhi =
  \begin{pmatrix}
    0\\Z
  \end{pmatrix}
  $. However,
  such an $X$ can also be written in the form
  \eqref{eq:identification_tangent_plane} with $\deltaPhi =
  \begin{pmatrix}
    A\\Z
  \end{pmatrix}
  $ for any
  anti-hermitian matrix $A$. The gauge condition $\Phi^{*}\deltaPhi=0$ forces
  $A$ to be zero, making $\deltaPhi$ unique.
  In more formal terms, the map $\Phi \mapsto \Phi \Phi^{*}$
  induces a principal bundle structure on the base space $\Mc_{\Nc}$ (the
  Grassmann manifold) with total space
  $\{\Phi \in \Cb^{\Nc \times \Nel}, \Phi^{*} \Phi = I_{\Nel}\}$ (the
  Stiefel manifold) and characteristic fiber ${\rm U}(N_{\rm el})$. This naturally
  splits the tangent space
  $\{\deltaPhi \in \Cb^{\Nc \times \Nel}, \Phi^{*} \deltaPhi + \deltaPhi^{*} \Phi = 0\}$ into
  the {\em vertical space} $\{\Phi A, \; A \mbox{ anti-hermitian}\}$, and a
  complementary {\em horizontal space}, which we take to be the orthogonal
  complement, $\{\deltaPhi \in \Cb^{\Nc \times \Nel} \; | \; \Phi^{*}\deltaPhi=0\}$.
\end{remark}

\medskip

The orbital formalism can be used to give a more concrete interpretation of the first-order optimality condition $R(P_*)=0$. Indeed, this condition can be rewritten as
\begin{align*}
  P_{*} H_{*} P^{\perp}_{*} = 0, \quad P^{\perp}_{*} H_{*} P_{*} = 0,
\end{align*}
from which it follows that $P_{*}$ and $H_{*}=H(P_{*})$ can be jointly diagonalized in an orthonormal basis:
\begin{equation}\label{eq:joint_diag}
  H_*\phi_{*n}=\lambda_{*n}\phi_{*n}, \quad  \cro{\phi_{*m},\phi_{*n}}=\delta_{mn}, \quad P_*= \sum_{n=1}^{\Nel} \ket{\phi_{*n}}\bra{\phi_{*n}}.
\end{equation}

In many applications, the orbitals $\phi_{*1},\dots,\phi_{*\Nel}$
spanning the range of $P_*$ (see
\eqref{eq:joint_diag}) are those corresponding to the lowest $\Nel$
eigenvalues of $H_*$. This is called the {\it Aufbau} principle in
physics and chemistry. This principle is always satisfied in the
(unrestricted) Hartree--Fock setting, and most of the times in the
Kohn--Sham setting. Under the {\it Aufbau} principle, we can assume
that the $\lambda_n$'s are ranked in nondecreasing order.
The orbitals $\phi_i$, $1 \lq i \lq \Nel$, are called occupied, and the orbitals $\phi_a$, $\Nel \lq a \lq \Nc$, are called virtual (it is customary to label the occupied orbitals by indices $i,j,k,l$, and virtual orbitals by indices $a,b,c,d$).
The operator $\bm{\Op}_*$ can be written explicitly using the tensor basis $\phi_{*m} \otimes \phi_{*n}$. We have indeed
\[
  \bm{\Op}_*=\sum_{i=1}^\Nel \sum_{a=\Nel+1}^{\Nc} (\lambda_a-\lambda_i) \prt{ \ket{\phi_{*i}\otimes\phi_{*a}}\bra{\phi_{*i}\otimes\phi_{*a}} + \ket{\phi_{*a}\otimes\phi_{*i}}\bra{\phi_{*a}\otimes\phi_{*i}}  },
\]
and it follows that the lowest eigenvalue of the restriction of $\bm{\Op}_*$ to $\Tc_{P_*}\Mc_\Nc$ is $\lambda_{\Nel+1}-\lambda_\Nel \gq 0$. The operator $\bm{\Op}_*$ is therefore positive on $\Tc_{P_*}\Mc_\Nc$, and coercive if there is an energy gap between the $\Nel^{\rm th}$ and $(\Nel+1)^{\rm st}$ eigenvalues of $H_*$ (see e.g.~\cite{cancesConvergenceAnalysisDirect2021}).

\subsection{Metrics on the tangent space}\label{sec:metrics}
The isomorphism between $X = \Phi \deltaPhi^{*} + \deltaPhi \Phi^{*} \in \Tc_{\Phi \Phi^{*}}\Mc_{\Nc}$ and the set of orbital
variations $\deltaPhi \in \Cb^{\Nc \times \Nel}$ with $\Phi^{*}\deltaPhi =
0$ is unitary under the Frobenius inner product up to a factor of $2$:
$
  \normF{X}^{2} = 2 \normF{\deltaPhi}^{2}
 $.

In practice, it is often advantageous to work using different inner
products. This is in particular the case for partial differential
equations involving unbounded operators, where using Sobolev-type
metrics better respects the natural analytic structure of the problem
and therefore allows for better bounds, compare e.g. the results of
(5.34) and (5.35) on Figure 4
in~\cite{cancesGuaranteedPosterioriBounds2020}. To that end, consider
a metric on $\Cb^{\Nc}$ given by
\begin{align*}
  \cro{\deltaphi_{1}, \deltaphi_{2}}_{T} = \cro{\deltaphi_{1},  T \deltaphi_{2}}.
\end{align*}
Here $T$ is a coercive Hermitian operator on $\Cb^\Nc$ representing the metric; for instance, taking
$T$ to be a discretization of the operator $1-\Delta$ we recover the
classical Sobolev $\H^{1}$ norm.
A basic problem is that the projection $P^\perp$ onto the orthogonal complement of
$\Ran(P)$ does not necessarily commute with $T$. As a
result, there are various nonequivalent ways to lift this metric to
one on the tangent space $\Tc_{\Phi \Phi^{*}}\Mc_{\Nc}$. We select here
the computationally simplest. The operator
\begin{align}  \label{eq:def_M}
  M = {P^\perp}T^{1/2}{P^\perp}T^{1/2}{P^\perp}
\end{align}
is positive definite on the subspace $\Ran(P)^{\perp}$ of $\Cb^{\Nc}$, and
induces a metric $\cro{\deltaphi_{1}, M \deltaphi_{2}}$ on that space.
The point of this formulation is to make it easy to compute
$M^{1/2} = {P^\perp}T^{1/2}{P^\perp}$. Note that, since ${P^\perp}$  and
$T$ do not commute,
$M^{-1/2} \neq {P^\perp}T^{-1/2}{P^\perp}$. However,
${P^\perp}T^{-1/2}{P^\perp} M^{1/2}$ is well-conditioned, so that computing
the action of $M^{-1/2}$ on a vector can be performed efficiently by an
iterative algorithm involving repeated applications of the
operators $T^{1/2}$ and $T^{-1/2}$. The same holds for $M^{-1}$.
Furthermore, practical numerical results are typically not very
sensitive to these issues, so that other (nonequivalent) reasonable alternatives to \eqref{eq:def_M} yield similar results.

The metric on $\Ran(P)^{\perp}$ immediately induces a metric on $\Tc_{\Phi
  \Phi^{*}}\Mc_{\Nc}$ given by, in the orbital representation associated with $\Phi$,
\begin{align*}
  \cro{\deltaPhi_{1},\deltaPhi_{2}}_{{\bm M}} = \Re\prt{\Tr(\deltaPhi_{1}^{*} M \deltaPhi_{2})}
  = \sum_{i=1}^{\Nel} \Re\prt{\cro{\deltaphi_{1,i}, M \deltaphi_{2,i}}},
\end{align*}
for
$\deltaPhi_{1}=(\deltaphi_{1,i})_{1\lq i\lq \Nel},\deltaPhi_{2}=(\deltaphi_{2,i})_{1\lq i\lq \Nel}$. This defines an
operator ${\bm M}$ on $\Tc_{\Phi \Phi^{*}}\Mc_{\Nc}$ through the
relationship ${\bm M} X \simeq_\Phi (M \deltaphi_{i})_{1\lq i\lq \Nel}$ when $X \simeq_\Phi (\deltaphi_{i})_{1\lq i\lq \Nel}$.
Similarly to $M$, we can compute powers and inverses of ${\bm M}$
easily.

This formalism has the disadvantage that the same metric is used for
every orbital variation. In practice this may not be sensible, as
different orbitals can correspond to different energy ranges.
Therefore we slightly modify the above formalism by applying a
different metric on each individual orbital variation, following
standard practice used in preconditioners for plane-wave density
functional theory~\cite{payneIterativeMinimizationTechniques1992}. Introducing a family $(T_1,\dots,T_\Nel)$ of coercive Hermitian operators on $\Cb^\Nel$, we set
\begin{equation}\label{eq:defMi}
  M_{i} \coloneqq {P^\perp}T^{1/2}_{i}{P^\perp}T^{1/2}_{i}{P^\perp} \quad \mbox{and} \quad  {\bm M} X \simeq_\Phi
  (M_{i} \deltaphi_{i})_{1\lq i\lq \Nel}.
\end{equation}

\subsection{Correspondence rules}
As explained above, the density matrices will be preferred for the mathematical analysis while orbitals will be used in practice in the numerical simulations. We summarize below the correspondence between the density matrix and molecular orbital formulations, and the practical way the different operators we introduced are computed.
For a given $P\in\Mc_{\Nc}$ and $(\phi_i)_{1\lq i\lq \Nel}$ the
associated set of occupied orbitals, there holds
\begin{equation*}
    \small
    \begin{array}{rclc}
        & & \\
        P \in \Mc_{\Nc}  & \leftrightarrow & \Phi=(\phi_1|\cdots|\phi_{\Nel}) \in \Cb^{\Nc\times\Nel} \mbox{ s.t. } P = \Phi\Phi^* & \mbox{(state)}, \\ & & \\
        \deltaP \in \Tc_P\Mc_{\Nc} &\leftrightarrow & \deltaPhi\in\Cb^{\Nc\times\Nel}
        \mbox{ s.t. } \Phi^*\deltaPhi=0& \mbox{(perturbation)},  \\ & & \\
        R(P) = [P,[P,H(P)]] & \leftrightarrow & (r_1|\cdots|r_{\Nel}) = P^\perp H(P) \Phi & \mbox{(residual)}, \\ & & \\
        \bm{\Op}(P)(X)  &\leftrightarrow & P^{\perp}(H \deltaPhi-\deltaPhi(\Phi^{*} H\Phi)) & \mbox{(see \eqref{eq:Omega_orbitals})}, \\ & & \\
        \bm{K}(P)(X) &\leftrightarrow & P^{\perp}(\delta H \phi_i)_{1\lq i\lq\Nel}  & \mbox{(see \eqref{eq:K_orbitals})}, \\ & & \\
        \bm{M}^{s}\deltaP &\leftrightarrow & \prt{M_i^s\deltaphi_i}_{1\lq i \lq \Nel} \text{ for } s=-1,-1/2,1/2,1 & \mbox{(see \eqref{eq:defMi})}. \\
        & & \\
    \end{array}
\end{equation*}

\section{The periodic Kohn--Sham problem}
\label{sec:kohn_sham}
\subsection{The continuous problem}
We consider an $\Rc$-periodic system, $\Rc$ being a Bravais lattice
with unit cell $\Gamma$ and reciprocal lattice $\Rc^*$ (the set of
vectors ${G}$ such that ${G} \cdot {R} \in 2\pi \mathbb Z$
for all ${R} \in \Rc$). For the sake of simplicity, we present
here the formalism for the (artificial) Kohn--Sham model for a finite
system of $\Nel$ electrons on the unit cell $\Gamma$ with
periodic boundary conditions. This is distinct from the more physical
periodic Kohn--Sham problem for an infinite crystal with $\Nel$
electrons by unit cell, which is usually treated by using the
supercell approach and Bloch
theorem. Practical computations are performed for the latter model
using Monkhorst-Pack Brillouin zone sampling
\cite{monkhorstSpecialPointsBrillouinzone1976} (see also~\cite{cancesNumericalQuadratureBrillouin2020} for a mathematical analysis of this method). The mathematical framework is very similar, with additional sums over $k$ points.

At the continuous level, a Kohn--Sham state is described by a density matrix $\gamma$, a rank-$\Nel$ orthogonal projector acting on the space
$\L^{2}_{\#}$ of square integrable periodic functions. Ignoring constant terms modeling interactions between ions (\ie atomic nuclear and frozen core electrons), the Kohn--Sham energy of $\gamma$ is given by $E^{\rm KS}(\gamma) = \Tr(h_0\gamma) + E^{\rm Hxc}(\rho_\gamma)$ (the superscript Hxc stands for Hartree-exchange-correlation), with
\[
  h_{0} = - \frac 1 2 \Delta + v_{\rm loc} + v_{\rm nloc},\quad
  E^{\rm Hxc}(\rho) =  \int_{\Gamma} \left( \frac 1 2 \rho V_{\rm H}(\rho)({x}) + e_{\rm xc}(\rho({x}))\right) \d {x}.
\]
In the above expressions, $\rho_{\gamma}$ is the density associated with the trace-class operator $\gamma$ (formally $\rho_\gamma({x})=\gamma({x},{x})$ where $\gamma({x},{x}')$ is the integral kernel of $\gamma$), $e_{\rm xc}:\Rb_+ \to \Rb$ a given
exchange-correlation energy, and $V_{\rm H}(\rho)$ the Hartree potential, defined as the unique
periodic solution with zero mean of the Poisson equation $-\Delta V_{\rm H}(\rho) = 4\pi \prt{ \rho - \fint_\Gamma \rho }$. In the pseudopotential approximation that we use in our numerical
results, $v_{\rm loc}$ is a local potential given by
\begin{equation}\label{eq:pot_loc}
  \Forall x \in \Rb^3, \quad V_{\rm{loc}}(x) \coloneqq \sum_{R\in\Rc}\sum_{j =1}^{N_{\rm at}} v_{\rm{loc}}^j\prt{x-(X_j+R)},
\end{equation}
and $v_{\rm nloc}$ a nonlocal potential in Kleinmann-Bylander~\cite{kleinmanEfficaciousFormModel1982} form given by
\begin{equation}\label{eq:pot_nonloc}
  V_{\rm{nloc}} \coloneqq \sum_{R\in\Rc}\sum_{j=1}^{N_{\rm at}}
  \sum_{a,b = 1}^{n_{\text{proj}, j}} \ket{p^j_{a}(\cdot-(X_j+R))}C^j_{ab}\bra{p^j_{b}(\cdot-(X_j+R))},
\end{equation}
where $N_{\rm at}$ is the number of atoms in $\Gamma$,
the ${X}_{j}$'s are the positions of the atoms inside the unit
cell $\Gamma$, $v_{\rm loc}^{j} : \Rb^{3} \to \Rb$ is a local radial potential,
$n_{\text{proj},j}$ denotes the number of projectors for atom $j$,
and the $p^{j}_{ab} : \Rb^{3} \to \Cb$ are given smooth functions. We use in
particular the Goedecker--Teter--Hutter (GTH) pseudopotentials
\cite{goedecker1996separable,hartwigsenRelativisticSeparableDualspace1998}
whose functional forms for the $v_{\rm loc}^{j}$ and $p^{j}_{ab}$ are
analytic ($v_{\rm loc}^{j}$ is a radial Gaussian function multiplied by a radial polynomial,
and $p^{j}_{ab}$ is a radial Gaussian function multiplied by a solid spherical harmonics).

The Kohn--Sham Hamiltonian associated to a density matrix $\gamma$ is given by
\begin{align*}
  h_\gamma = h_0 + V_{\rm H}(\rho_{\gamma}) + e'_{\rm xc}(\rho_{\gamma}),
\end{align*}
where $V_{\rm H}(\rho_{\gamma}) $ and $e'_{\rm xc}(\rho_{\gamma})$ are
interpreted as local (multiplication) operators. Similarly, we have
\begin{align*}
  D^2_\gamma (E^{\rm Hxc}(\rho_\gamma)) \cdot Q = V_{\rm H}(\rho_{Q}) + e''_{\rm xc}(\rho_{\gamma}) \rho_{Q}.
\end{align*}

\begin{remark}[Spin]
  \label{rem:spin}
  The expressions above are given for a system of ``spinless
  electrons'' to accomodate the simple geometrical formalism of
  Section~\ref{sec:sec2}. Real systems (and the numerical simulations
  we perform in the following sections) include spin; in this case,
  the energy is $E^{\rm KS}(\gamma) = 2\Tr(h_0\gamma) + E^{\rm
    Hxc}(\rho_\gamma)$, where $\rho_{\gamma}(x) = 2 \gamma(x,x)$,
  $\nabla E^{\rm KS} (\gamma) = 2(h_0 + V_{\rm H}(\rho_{\gamma}) + e'_{\rm
    xc}(\rho_{\gamma}))$ and $D^2 (E^{\rm KS}(\gamma)) \cdot Q
  = 4 V_{\rm H}(\rho_{Q}) + 4 e''_{\rm xc}(\rho_{\gamma}) \rho_{Q}.$
\end{remark}

\subsection{Discretization}\label{sec:discr}

For each vector $G$ of the reciprocal lattice $\Rc^*$, we denote by $e_{G}$ the Fourier mode with wave-vector $G$:
\[
  \Forall x\in \Rb^3, \quad e_{G}(x) \coloneqq \frac{1}{\sqrt{\abs{\Gamma}}} \exp\prt{\i G
    \cdot x}
\]
where $\abs{\Gamma}$ is the Lebesgue measure of the unit cell $\Gamma$. The family
$(e_{G})_{G\in\Rc^*}$ is an orthonormal basis of $\L^2_\#$, the space of locally square integrable $\Rc$-periodic functions (and an orthogonal basis of the $\Rc$-periodic Sobolev space $\H^s_\#$, endowed with its usual inner product, for any $s \in \Rb$). In the so-called plane-wave discretization methods, the Kohn--Sham model is discretized using the finite-dimensional approximation spaces
\[
  \Xc_\Ecut \coloneqq \Span\set{e_{G}, \; G\in\Rc^* \; \middle| \;
    \frac{1}{2}\abs{G}^2 \lq \Ecut},
\]
where $\Ecut>0$ is a given energy cut-off chosen by the user.

\medskip

The connection with the formalism introduced in \autoref{sec:sec2} is the following:
\begin{itemize}
  \item we choose a large reference energy cut-off $E_{\rm cut,ref}$ and set
    \[
      \Nc\coloneqq{\rm dim}(\Xc_{E_{\rm cut,ref}})=\#\set{{G} \in \Rc^* \; | \; \frac 12 |{G}|^2 \lq E_{\rm cut,ref}};
    \]
  \item we identify $\Xc_{E_{\rm cut,ref}}$ with $\Cb^\Nc$ by labelling the reciprocal lattice vectors from $1$ to $\Nc$ in such a way that for all $1 \lq i < j \lq \Nc$, $|{G}_i| \lq |{G}_j|$;
  \item the set of rank-$\Nel$ orthogonal projectors $\gamma$ on $\L^2_\#$ such that $\Ran(\gamma) \subset \Xc_{E_{\rm cut,ref}}$ can then be identified with the manifold $\Mc_\Nc$ defined in \eqref{eq:pb} through the mapping
    \[
      \gamma = \sum_{i,j=1}^\Nc P_{ij} \ket{e_{{G}_i}}\bra{e_{{G}_j}};
    \]
  \item the noninteracting Hamiltonian matrix $H_0 \in \Cb^{\Nc \times \Nc}$ has entries
    \[
      [H_0]_{ij} = \bak{e_{{G}_i}}{h_0}{e_{{G}_j}}_{\L^2_\#},
    \]
    and the nonlinear component of the energy $E_{\rm nl}:\Hc \to \Rb$ is any $C^2$-extension of the function defined on $\Mc_\Nc$ by
    \[
     E_{\rm nl}(P)=E^{\rm Hxc}(\rho_P) \quad \mbox{where} \quad \rho_P({x}) 
     = |\Gamma|^{-1/2} \sum_{i,j=1}^\Nc P_{ij} e_{G_j-G_i}({x}).
    \]
\end{itemize}

The entries of the core Hamiltonian matrix can be computed explicitly:
$$
  [H_0]_{ij} =  \frac{\abs{G_i}^2}{2} \delta_{i,j} + [V_{\rm loc}]_{ij} + [V_{\rm nloc}]_{ij}
  $$
  $$
  \mbox{with} \quad   [V_{\rm loc}]_{ij} = \bak{e_{G_i}}{V_{\rm loc}}{e_{G_j}}_{\L^2_\#} \quad \mbox{and} \quad  [V_{\rm nloc}]_{ij} = \bak{e_{G_i}}{V_{\rm nloc}}{e_{G_j}}_{\L^2_\#},
$$
where the above inner products can be computed exactly through the Fourier
transforms of the $v_{\rm loc}^{j}$ and $p_{ab}^{j}$ (known exactly for GTH pseudopotentials). Note also that the
density $\rho_{P}$ can be expanded on a finite number of Fourier modes and can therefore be easily stored in memory. Since the
Poisson equation is trivially solvable in the plane-wave basis, this
enables the exact computation of the Hartree energy. The
exchange-correlation energy however cannot be computed explicitly, and is
approximated using numerical quadrature. In all the numerical results,
we select the parameters of this numerical quadrature such that it
does not affect too much the results, see Remark~\ref{rmk:sampl}.

\subsection{Forces}
\label{sec:forces}
The total ground-state energy depends on the atomic positions $ \mathfrak{X} =
(X_{j})_{1\lq j \lq N_{\rm at}}$ both explicitly (ion-ion interaction energy and ion-electron interaction potentials $V_{\rm loc}$ and $V_{\rm nloc}$) and
through the fact that the solution $P_{*}$ depends on $\mathfrak{X} $:
\begin{align*}
  {\mathcal E}({\mathfrak{X}}) = E({\mathfrak{X}}, P_{*}({\mathfrak{X}})).
\end{align*}
The force acting on atom $j$ is defined as $F_{j}(\mathfrak{X}) =
-\nabla_{X_j} {\mathcal E}(\mathfrak{X})$.  Because of the Hellman--Feynman theorem, the term involving
the derivative of $P_{*}$ with respect to $X_{j}$ vanishes~\cite{martinElectronicStructureBasic2004}, and the final result is
\begin{align} \label{eq:def_Fj}
  F_{j} = - \Tr\prt{(\nabla_{{X}_j} (V_{\rm loc}+V_{\rm nloc})) P_{*}}.
\end{align}
This involves the partial derivatives of the matrix elements of
$V_{\rm loc} + V_{\rm nloc}$ with respect to the atomic positions,
which can be computed analytically from \eqref{eq:pot_loc} and \eqref{eq:pot_nonloc}.

\subsection{Numerical setup}
For all the computations and examples on silicon, we use the DFTK
software~\cite{herbstDFTKJulianApproach2021} within the LDA
approximation, with Teter 93 exchange-correlation functional
\cite{goedecker1996separable} and a $2\times2\times2$ ${k}$-point
grid, and a reference solution computed with $\Ecutref = 125$ Ha, to
which we compare results obtained with smaller values of $\Ecut$. We checked that $\Ecutref = 125$ Ha was a high enough energy cut-off to have fully converged results, up to the accuracy we need to test our numerical methods. We use the
usual periodic lattice for the FCC phase of silicon, with lattice constant
$a=10.26$ bohrs, close to the equilibrium configuration. All results are expressed in atomic units: energies are in hartree and forces are in hartree/bohr. Note that the discretization grid of the Brillouin zone is not fine enough to have fully converged results, but is still sufficient to illustrate our points. Note also that the same results are observed for semilocal functionals, such as PBE-GGA \cite{perdewGeneralizedGradientApproximation1996}. Other functionals, such as meta-GGA and hybrid functionals, are out of the scope of this paper.

The two atoms of silicon inside a cell are placed at first at their equilibrium positions with fractional coordinates $(-\frac 1 8, -\frac 1 8, -\frac 1 8)$ and
$(\frac 1 8, \frac 1 8, \frac 1 8)$, and then the second one is slightly displaced by $\frac 1 {20} (0.24, -0.33, 0.12)$
to get nonzero interatomic forces.

\medskip

The discretized Kohn--Sham equations are solved by a standard SCF procedure. The main computational bottleneck is the
partial diagonalization of the mean-field Hamiltonian at each SCF step. This is done using an iterative eigenvalue solver,
which only requires applying mean-field Hamiltonian matrices to a set of $\Nel$ trial orbitals and simple operations on vectors. In a plane-wave basis set of size $N_b$, the former operation can be done efficiently through the use of the fast Fourier transform
for a total cost of $O(\Nel N_{b}(\log N_{b} + \sum_{j} n_{\text{proj},j}))$. We
refer to~\cite{martinElectronicStructureBasic2004} for more details. The
application of the super-operators ${\bm \Op}$ and ${\bm K}$
to a set of $\Nel$ orbital variations (see \eqref{eq:Omega_orbitals} and \eqref{eq:K_orbitals}) involves additional linear
algebra operations, for an additional cost of $O(\Nel^{2}(N_{b} + \Nel))$.

\medskip

In this setting, the reference values for the energy is $E_* = -7.838 $ Ha and the interatomic forces are, in hartree/bohr,
\[
F_* = \begin{bmatrix}
 -0.0656 &  0.0656 \\
  0.0619 & -0.0619 \\
 -0.0352 &  0.0352 \\
\end{bmatrix},
\]
where the first column are the forces acting on the first atom in each direction, and the second column are the forces acting on the second atom.

\section{A first error bound using linearization}
\label{sec:sec3}

Now that the mathematical and numerical frameworks are laid down, we
turn to the estimation of the error between the reference solution
computed with a large energy cut-off $\Ecutref$ and approximations thereof. We first start by
deriving a linearization estimate and illustrating numerically its
applicability. We then propose a very coarse bound on the error on the
density matrix and the forces, based on the (expensive) evaluation of an operator norm. We will show in the next section how to improve this bound.

\subsection{Linearization in the asymptotic regime}

We assume that $P_*$ is a nondegenerate local minimizer of $E$ in the
sense that there exists $\eta > 0$ such that
$\bm{\Op}_* + \bm{K}_* \gq \eta$ on the tangent space $\Tc_{P_*}\Mc_{\Nc}$. This implies in particular that $\bm{\Op}_* + \bm{K}_*$ is invertible on the invariant subspace $\Tc_{P_*}\Mc_\Nc$.

Recall that for any trial density matrix $P \in \Mc_\Nc$, the residual of the problem is
\begin{align*}
  R(P) = \bm{\Pi}_{P} H(P) = [P,[P,H(P)]] \in \Tc_P\Mc_\Nc,
\end{align*}
so that $R$ defines a smooth vector field on $\Mc_{\Nc}$ (a section of the tangent bundle $\Tc_P\Mc_\Nc$) which vanishes at $P_*$. For $P \in \Mc_\Nc$ in the vicinity of $P_*$, we have
\begin{equation}\label{eq:P-P*}
  P-P_*= \bm{\Pi}_{P_*}(P-P_*)+O\prt{\normF{P-P_*}^2} = \bm{\Pi}_{P}(P-P_*)+O\prt{\normF{P-P_*}^2}.
\end{equation}
It follows from the definitions \eqref{eq:defK*}-\eqref{eq:defOmega*} of $\bm{\Op}_*$ and
$\bm{K}_*$ that $\bm{\Op}_* + \bm{K}_*$ is the Jacobian of the map
$P \mapsto R(P)$ at $P_*$.  Therefore, the optimality condition $R(P_*)=0$ and the
above expansions yield, for all $P \in \Mc_\Nc$ close enough to $P_{*}$,
\begin{equation}\label{eq:RP}
    \begin{split}
      R(P) &=  (\bm{\Op}_* + \bm{K}_*)\bm{\Pi}_{P_*}(P-P_*) + O\prt{\normF{P-P_*}^2}\\
      &= (\bm{\Op}(P) + \bm{K}(P))\bm{\Pi}_{P} (P-P_*) + O\prt{\normF{P-P_*}^2}.
    \end{split}
\end{equation}
By continuity, $\bm{\Op}(P) + \bm{K}(P) \gq \frac \eta 2$ on the
tangent space $\Tc_{P}\Mc_{\Nc}$ for $P \in \Mc_\Nc$ close enough to
$P_*$, so that the restriction of the super-operator $\bm{\Op}(P) +
\bm{K}(P)$ to the invariant subspace $\Tc_{P}\Mc_{\Nc}$ is
self-adjoint and invertible. Using again \eqref{eq:P-P*} and the
fact that $R(P) \in \Tc_P\Mc_\Nc$, we obtain, after inversion of
$\bm{\Op}(P) + \bm{K}(P)$ on the tangent space,
\begin{equation}
  \label{eq:linearize}
  \boxed{P-P_* =  ((\bm{\Op}(P) + \bm{K}(P))|_{\Tc_P\Mc_\Nc})^{-1}R(P) + O\prt{\normF{P-P_*}^2}.}
\end{equation}
This error-residual equation is the analog in our case of the
linearization~\eqref{eq:linearization}, which identifies the super-operator
$\bm{\Op}(P)+\bm{K}(P)$ as the fundamental object in our study.

Based on this expansion, we can formulate the Newton
algorithm to solve the equation $R(P_*)=0$:
\[
  P^{k+1} = \mathfrak{R}_{P^k}\prt{P^{k}-(\bm{\Op}^k + \bm{K}^k)^{-1}R(P^k)},
\]
where $\bm{\Op}^k\coloneqq\bm{\Op}(P^k)|_{\Tc_{P_k}\Mc_\Nc}$ and
$\bm{K}^k\coloneqq\bm{K}(P^k)|_{\Tc_{P_k}\Mc_\Nc}$ and $\mathfrak{R}$ is a suitable
retraction on $\Mc_\Nc$. A possible
retraction is given in~\cite{cancesConvergenceAnalysisDirect2021}.
This Newton algorithm is expensive in practice, as it requires to solve
iteratively a linear system; the cost of a Newton step is comparable to that of a full
self-consistent field cycle. It is however a useful theoretical
tool, and a starting point for further analysis and approximations.

\bigskip

To check the validity of the linearization \eqref{eq:linearize}, we
focus on three quantities of interest: the ground state energy,
the ground-density density, and the
interatomic forces acting on the two atoms in $\Gamma$. The reference values $E_*$, $\rho_*$ and $F_*$ of these QoIs are those obtained with the very large energy cut-off $E_{\rm cut,ref}=125$ Ha, defining a ``fine grid'' in real space via the discrete Fourier transform. For $E_{\rm cut} < E_{\rm cut,ref}$ defining a ``coarse grid'' in real space, we compute two approximations of the three QoIs:
\begin{enumerate}
  \item $E_{\rm SCF}$, $\rho_{\rm SCF}$ and $F_{\rm SCF}$ denote the approximations obtained from the variational solution of the Kohn--Sham problem on the coarse grid;
  \item $E_{\rm Newton}$, $\rho_{\rm Newton}$ and $F_{\rm Newton}$ denote the ones computed from the Kohn--Sham state obtained by one Newton step on the fine grid, starting from the variational solution of the Kohn--Sham problem on the coarse grid: as the SCF is converged on the coarse grid, we perform the Newton step on the fine grid in order to improve the approximation of $P_*$. That is, if $P$ is the variational solution on the coarse grid, $P-((\bm{\Op}(P) + \bm{K}(P))|_{\Tc_P\Mc_\Nc})^{-1}R(P)$ is a much better approximation of $P_*$ (for the metrics adapted to the chosen three QoIs).
\end{enumerate}
The errors between these approximations and the reference values are
plotted in \autoref{fig:linearization} as functions of $E_{\rm
  cut}$. The errors on the ground-state density are measured with the
$\L^2_\#$ metric, while the errors on the forces are measured with the
Euclidean metric on $\Rb^{3 \times 2}$.

For the simple case of a silicon crystal at the LDA level of theory, the linearization works very well, even for
very small values of $\Ecut$'s of the order of $5$ Ha. Indeed the Kohn--Sham ground-state obtained by variational approximation on a coarse grid is significantly improved by one Newton step: the errors on the QoIs obtained with the latter are orders of magnitude smaller than the ones obtained with the former.

\begin{figure}[ !ht]
  \centering
  \includegraphics[width=0.32\linewidth]{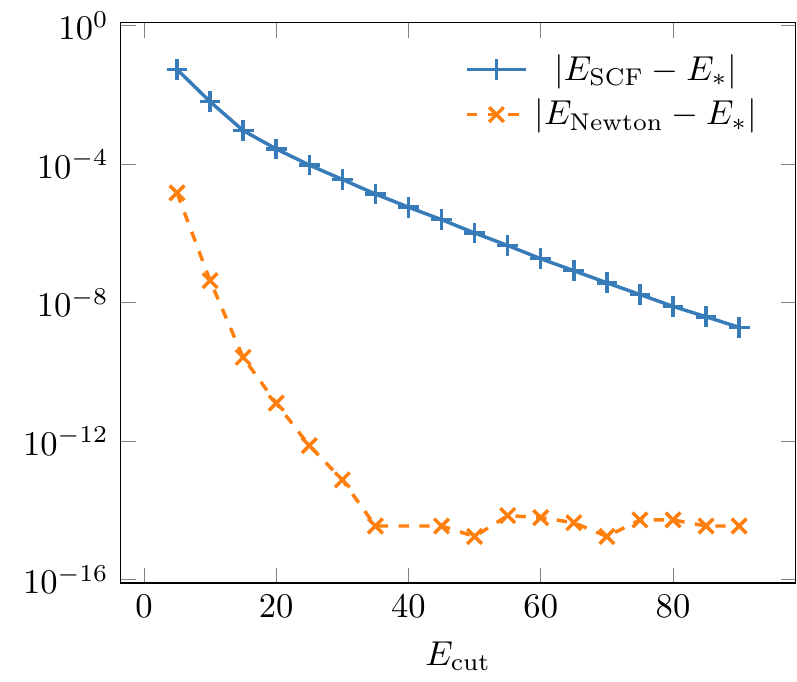}
  \hfill
  \includegraphics[width=0.32\linewidth]{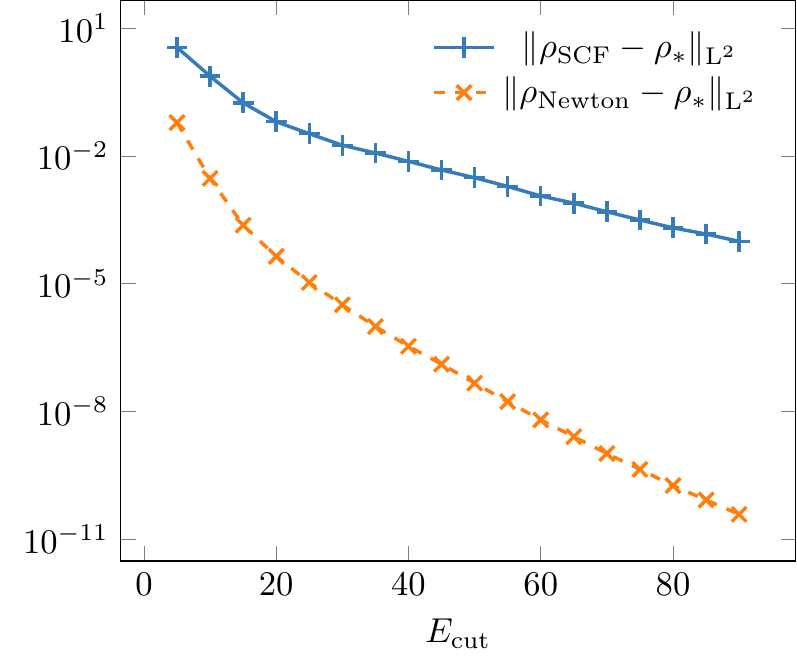}
  \hfill
  \includegraphics[width=0.32\linewidth]{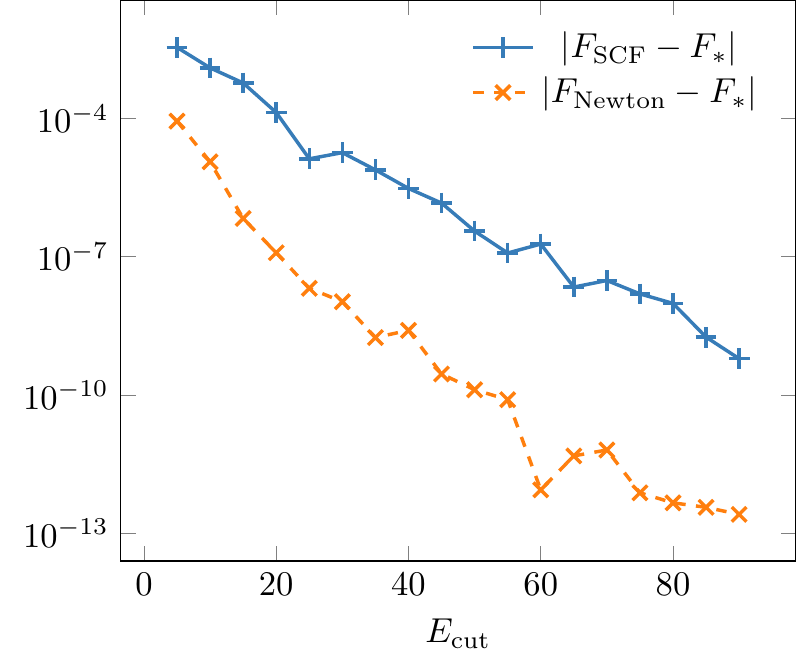}
  \caption{Errors for the ground-state energy (hartree), ground-state density and interatomic forces (hartree/bohr) for Si as a function of
    $\Ecut$, for both the variational solution of the Kohn--Sham problem on the coarse grid defined by $E_{\rm cut}$ (solid line) and the post-processed solution obtained with one Newton step on the fine grid (dashed line). This shows that the linearization approximation is excellent, even for energy cut-offs as low as $\Ecut=5$ Ha.}
  \label{fig:linearization}
\end{figure}

\subsection{A simple error bound based on operator norms}
\label{sec:simple_error_bound}

From \eqref{eq:linearize} one can extract an error bound:
\begin{equation}
    \label{eq:badbound}
    \begin{split}
      \norm{P-P_*} &\approx \norm{\bm{\Pi}_{P}(P-P_*)} \\ &\lq \norm{((\bm{\Op}(P) + \bm{K}(P))|_{\Tc_P\Mc_\Nc})^{-1}}_{\rm op} \norm{R(P)} \quad \mbox{(+ h.o.t.),}
    \end{split}
\end{equation}
where $\norm{\cdot}_{\rm op}$ is the (super-)operator norm associated with the chosen norm $\norm{\cdot}$ on $\Hc$.
This bound is not guaranteed, but the
results in \autoref{fig:linearization} suggest that it is very close
to be guaranteed. Guaranteeing this bound could be done, provided that one could
bound the higher-order terms rigorously
\cite{schmidtRigorousEffectiveAposteriori2020}; this is an interesting
prospect, but lies outside the scope of this paper.
To test the accuracy of this bound for a specific norm on $\Hc$,  we would need to estimate the corresponding operator
norm of the Hermitian operator $((\bm{\Op}(P) + \bm{K}(P))_{\Tc_P\Mc_\Nc})^{-1}$ for all the $P$'s we are considering.
In order to lower the computational burden, we consider instead the bound
\begin{align}
  \label{eq:badbound2}
  \norm{P-P_*} \lq \norm{((\bm{\Op}_* + \bm{K}_*)|_{\Tc_{P_*}\Mc_\Nc})^{-1}}_{\rm op} \norm{R(P)} \quad \mbox{(+ h.o.t.).}
\end{align}
This enables us to compute the operator norm $\norm{((\bm{\Op}_* + \bm{K}_*)|_{\Tc_{P_*}\Mc_\Nc})^{-1}}_{\rm op}$ only once, instead of computing it for every $P$. This is of course not accessible in practice, but we use it here for the sake of numerical experiment.
Moreover, we can consider that the bounds \eqref{eq:badbound} and \eqref{eq:badbound2} are almost equivalent since the results obtained in the previous section show that we are in the linear regime even for the lowest values of $E_{\rm cut}$ used in practice. The operator $(\bm{\Op}_* + \bm{K}_*)|_{\Tc_{P_*}\Mc_\Nc}$ is
Hermitian for the Frobenius inner product and, thus, the operator norm $\norm{((\bm{\Op}_* + \bm{K}_*)_{\Tc_{P_*}\Mc_\Nc})^{-1}}_{\rm op}$ corresponding to the Frobenius norm on $\Hc$ is equal to the inverse of the smallest eigenvalue of $(\bm{\Op}_* + \bm{K}_*)|_{\Tc_{P_*}\Mc_\Nc}$. Standard iterative eigenvalue solvers
for Hermitian operators can be used to compute this eigenvalue. We use here the LOBPCG algorithm
\cite{knyazevOptimalPreconditionedEigensolver2001}.

We can see on \autoref{fig:cs_bad} (left panel) that when choosing the Frobenius norm on $\Hc$, the bound \eqref{eq:badbound2} leads to very crude error estimates: the error is
overestimated by several orders of magnitude, and the bound becomes
worse and worse as $\Ecut$ increases. This issue is well-known in the analysis of partial differential
equations, where $\L^{2}$-type norms are not the natural ones to measure
the error on the solution or the residual. Instead, for the Kohn--Sham equations and other second-order elliptic problems, it is more relevant to measure the error $P-P_*$ in $\H^{1}$-type Sobolev norms (energy norms) and the residual $R(P)$ in $\H^{-1}$-type Sobolev norms (dual norms). The linear operator linking the two quantities (here  $(\bm \Op(P) + \bm K(P))|_{T_{P}\Mc_\Nc}$) is then expected to be a bounded isomorphism from the state error to the residual space for these norms. This suggests adapting the metrics on the tangent space $T_{P}\Mc_\Nc$ in which we measure the error $P-P_*$ (or more precisely the leading term $\bm{\Pi}_{P}(P-P_*)$) on the one hand, and the residual $R(P)$ on the other hand. Similar considerations
lead to the ``kinetic energy preconditioning'' used in practical
computations~\cite{payneIterativeMinimizationTechniques1992}.
Using the super-operator $\bm{M}$ on $\Tc_P\Mc_\Nc$ introduced in \eqref{eq:defMi}
with $T_i$ the diagonal operator on $\Cb^\Nc$ representing the operator $-\frac{1}{2}\Delta + t_i$ where $t_i=\frac 12 \norm{\nabla\phi_i}_{\L^2_\#}^2$ (kinetic energy of the $i^{\rm th}$ orbital), we obtain the bound
\begin{align}
  \label{eq:bound_P_M}
  \normF{\bm{M}^{1/2}\bm{\Pi}_{P}(P-P_*)} \!\! \lq \!\! \norm{\bm{M}^{1/2} ((\bm{\Op}(P) + \bm{K}(P))|_{\Tc_P\Mc_\Nc})^{-1} \bm{M}^{1/2}}_{\rm op} \normF{\bm M^{-1/2} R(P)}.
\end{align}

Here also, we lower the computational burden by replacing the first term in the RHS
by the asympotically equal quantity $\norm{\bm{M}_*^{1/2} ((\bm{\Op}_* + \bm{K}_*)|_{\Tc_{P_*}\Mc_\Nc})^{-1} \bm{M}_*^{1/2}}_{\rm op}$.
The results are shown in \autoref{fig:cs_bad} (central panel). This time,
the curves are almost parallel: the gap does not widen as $\Ecut$
increases.
However, the bound is still an overestimate
by more than one order of magnitude. This is due to the fact that
\[
  \norm{\bm{M}_*^{1/2}((\bm{\Op}_* + \bm{K}_*)|_{\Tc_{P_*}\Mc_\Nc})^{-1}\bm{M}_*^{1/2}}_{\rm op} \approx
  14.85
\]
for this system, while the residual $R(P)$ is supported only on high-frequency Fourier modes, on which the operator
$\bm{M}^{1/2} ((\bm{\Op}(P) + \bm{K}(P))|_{\Tc_P\Mc_\Nc})^{-1} \bm{M}^{1/2}$ is close to
identity. The latter statement is supported by Proposition~\ref{prop:err_res_cvg} in the appendix (see also the result~\cite[Proposition 5.10]{cancesGuaranteedPosterioriBounds2020} concerning the linear setting). Thus, $\normF{\bm M^{-1/2} R(P)}$ is a good approximation
of $\normF{\bm{M}^{1/2}\bm{\Pi}_{P}(P-P_*)}$, as shown on
\autoref{fig:cs_bad} (central panel).

\begin{figure}[ !ht]
  \centering
  \includegraphics[width=0.3\linewidth]{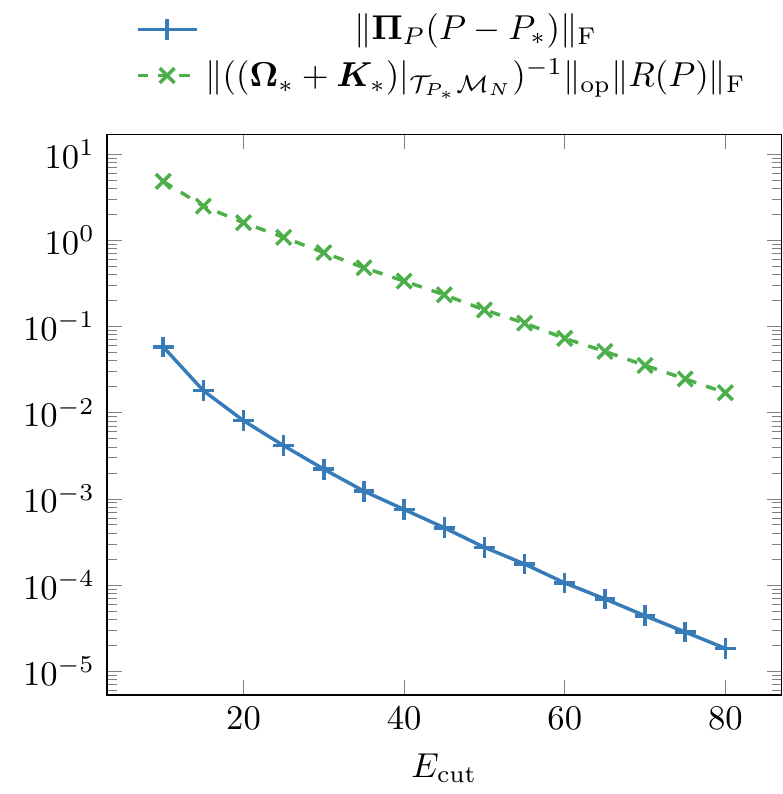}\hfill
  \includegraphics[width=0.3\linewidth]{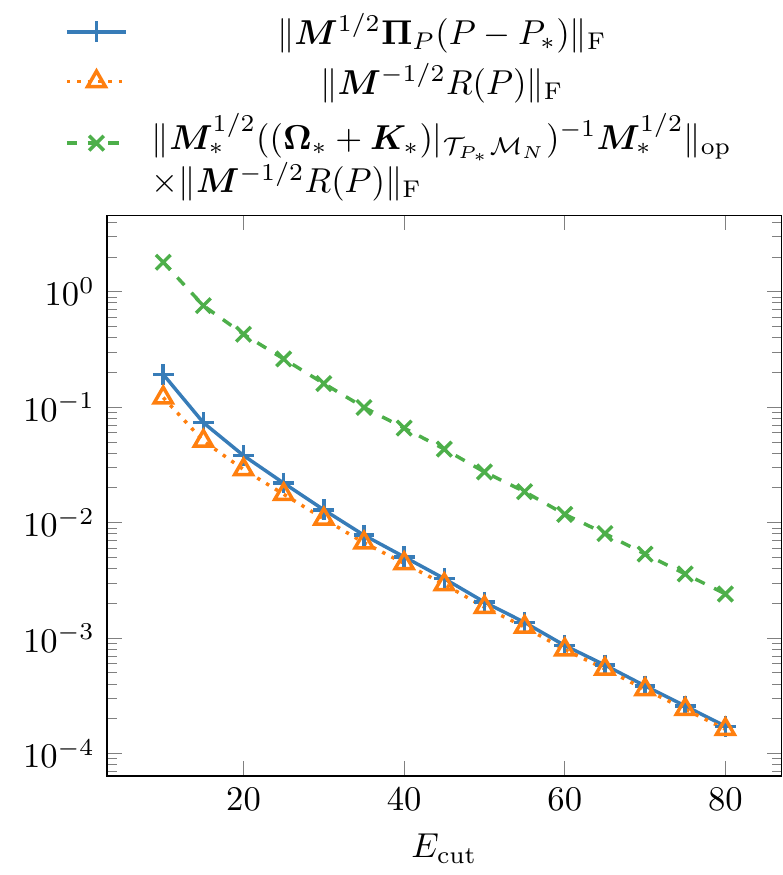}\hfill
  \includegraphics[width=0.3\linewidth]{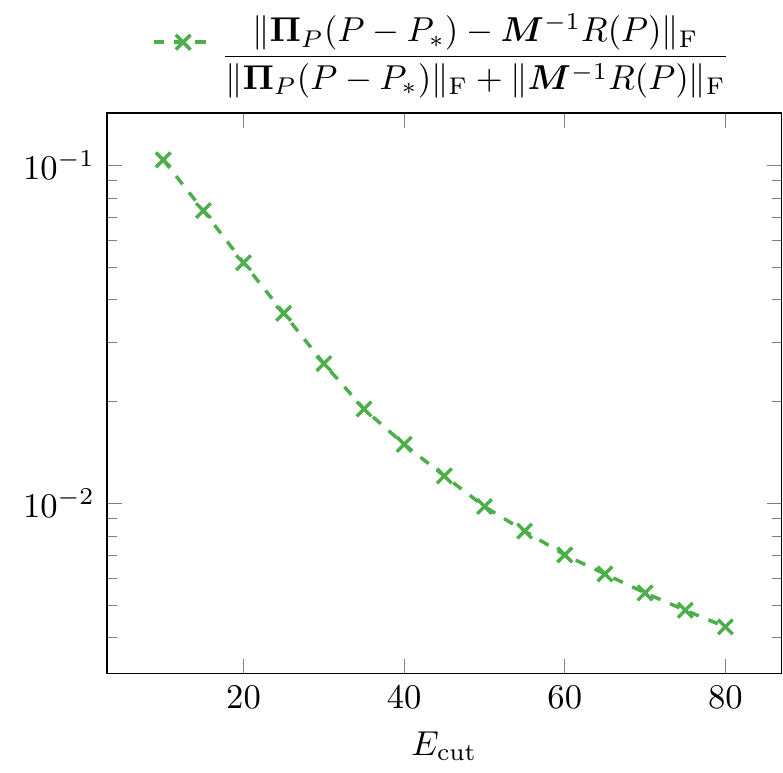}
  \caption{Error bounds for Si based on  \eqref{eq:badbound2} and \eqref{eq:bound_P_M}.
    Left: $\L^2$-norm; Center: $\H^1$-type norm; Right: relative error between $\bm{\Pi}_P(P-P*)$ and $\bm{M}^{-1}R(P)$. It holds
    $\norm{((\bm{\Op}_* + \bm{K}_*)|_{\Tc_{P_*}\Mc_\Nc})^{-1}}_{\rm op} \approx 11.23$
     and $\norm{\bm{M}_*^{1/2}((\bm{\Op}_* + \bm{K}_*)|_{\Tc_{P_*}\Mc_\Nc})^{-1}\bm{M}_*^{1/2}}_{\rm op} \approx 14.85$.  }
  \label{fig:cs_bad}
\end{figure}

\subsection{Error bounds on QoIs and applications to interatomic forces}\label{sec:estim_qoi}

Consider now a quantity of interest characterized by the smooth observable $A : \Mc_\Nc \to \Gc$, where $\Gc$ is a normed vector space (in particular, $\Gc=\Rb$ for real QoIs such as the ground state energy, $\Gc=\Rb^{3N_{\rm at}}$ for interatomic forces, and, e.g., $\Gc=\L^2_\#$ for the ground-state densities). For such a QoI, there holds for $P \in \Mc_\Nc$ in the vicinity of $P_*$,
\begin{equation}\label{eq:qoi_lin}
  A(P) - A_* = \d A(P) \cdot (\bm{\Pi}_{P}(P - P_*)) + {\rm h.o.t.},
\end{equation}
where $\d A(P) \in \Lc( \Tc_{P}\Mc_\Nc ; \Gc)$ is the derivative of $A$ at $P$. We thus obtain the bound
\begin{equation}\label{eq:qoi_bound}
  \norm{A(P) - A_*}_{\Gc} \lq \norm{\d A(P)}_{\Tc_P\Mc_\Nc \to \Gc} \norm{\bm{\Pi}_{P}(P-P_*)}_{\Tc_P\Mc_\Nc} \quad \mbox{(+ h.o.t.)}
\end{equation}
for given norms $\norm{\cdot}_\Gc$ and $\norm{\cdot}_{\Tc_P\Mc_\Nc}$ on~$\Gc$ and $\Tc_P\Mc_\Nc$ respectively, and associated operator norm $\norm{\cdot}_{\Tc_P\Mc_\Nc \to \Gc}$ on $\Lc( \Tc_{P}\Mc_\Nc ; \Gc)$.

Let us start with the simple case of the component of the force on atom $j$ along the direction $\alpha$ due to the local part of the pseudopotential. Since this QoI is scalar, we have $\Gc = \Rb$. Using \eqref{eq:def_Fj}, we get
\[
  F_{j,\alpha}^{\rm loc}(P) = -\Tr\prt{\frac{\partial V_{\rm loc}}{\partial X_{j,\alpha}}P}.
\]
Thus \eqref{eq:qoi_bound} becomes, using the Frobenius norm on $\Tc_{P}\Mc_\Nc$,
\begin{equation}\label{eq:fdir_bound}
  |F_{j,\alpha}^{\rm loc}(P) - F_{j,\alpha}^{\rm loc}(P_*)| \lq \normF{\bm{\Pi}_P\frac{\partial V_{\rm loc}}{\partial X_{j,\alpha}}} \normF{\bm{\Pi}_{P}(P-P_*)} \quad \mbox{(+ h.o.t.)},
\end{equation}
with
\begin{equation}\label{eq:fdir_diff}
  \bm{\Pi}_P\frac{\partial V_{\rm loc}}{\partial X_{j,\alpha}} \simeq_\Phi (1-\Phi\Phi^*)\prt{\frac{\partial V_{\rm loc}}{\partial X_{j,\alpha}}\Phi}.
\end{equation}
We plot in \autoref{fig:forces} (left panel) the bound
\eqref{eq:fdir_bound}.
The latter is pessimistic by more than three orders of magnitude, and its relative accuracy gets worse and worse as the cut-off energy increases.

\begin{figure}[ !ht]
  \hfill
  \includegraphics[width=0.4\linewidth]{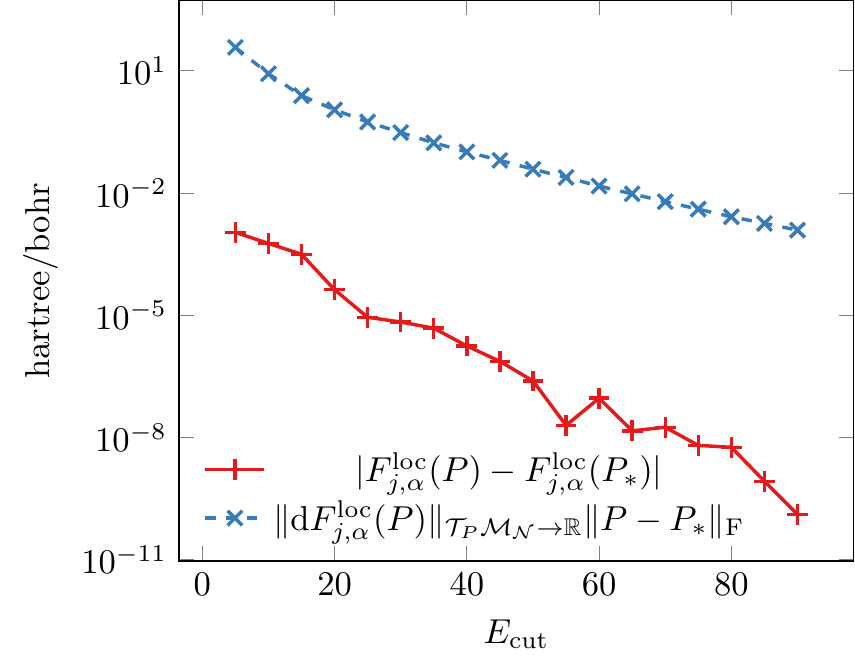}
  \hfill
  \includegraphics[width=0.4\linewidth]{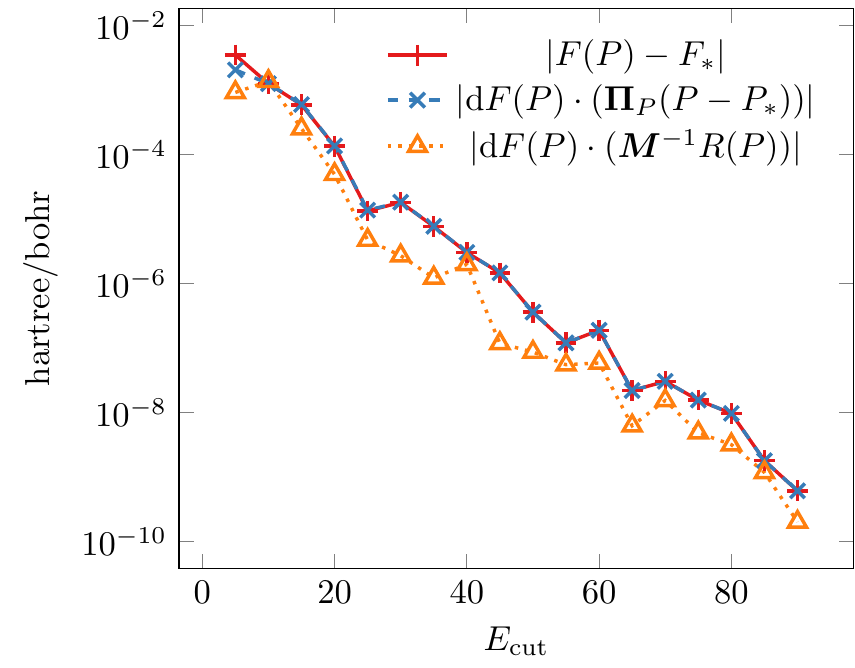}
  \hfill{~}
  \caption{Silicon. (Left panel) Inaccurate error bound \eqref{eq:fdir_bound} for the component of the force on atom $j =1$ along direction $\alpha = (1,0,0)$ due to the local part of the pseudopotential. (Right panel)
    Approximation of $\abs{F(P) - F_*}$ obtained by dropping the h.o.t. in the generic formula \eqref{eq:qoi_lin} and applying the derivative $\d F(P)$ either to the actual error $\bm{\Pi}_P(P-P_*)$ or
    the preconditioned residual $\bm{M}^{-1}R(P)$. The approximation $\d F(P)\cdot (\bm{\Pi}_P(P-P_*))$
    matches asymptotically the error  $F(P) - F_*$, validating again the rapid establishment
    of the linear regime. On the other hand, the approximation $\d F(P)\cdot(\bm{M}^{-1}R(P))$ does not match asymptotically.}
  \label{fig:forces}
\end{figure}

The bound \eqref{eq:fdir_bound} using the operator norm of $\d A(P)$ being very inaccurate, we tested another approach consisting in using directly \eqref{eq:qoi_lin} to evaluate the error on the QoI by applying the derivative $\d A(P)$ to a computable approximation of $\bm{\Pi}_P(P-P_*)$.
Relying on the results in the previous section showing that $\bm
M^{-1/2} R(P)$ is a good approximation of $\bm M^{1/2} \bm{\Pi}_P(P-P_*)$ in Frobenius norm,
it is tempting to replace $\bm{\Pi}_P(P-P_*)$ by $\bm M^{-1} R(P)$ in \eqref{eq:qoi_lin} and approximate $F(P) - F_*$ by $\d
F(P) \cdot (\bm M^{-1} R(P))$, this approximation being justified by \autoref{fig:cs_bad} (right panel). Indeed the continuous counterpart of the asymptotic equivalence between $\bm
M^{-1/2} R(P)$ and $\bm M^{1/2} \bm{\Pi}_P(P-P_*)$ for the Frobenius ($\L^2$-type) norm is that the preconditioned residual and the error on the density matrix are asymptotically equivalent in $\H^1$-type norms, while the derivative of the interatomic forces observable is continuous on $\H^1$-type spaces. This idea is tested in \autoref{fig:forces} (right panel). However, this leads to an underestimation of the error, although by a small factor.
The reason is that even if $P-P_*$ and $M^{-1} R(P)$ do match asymptotically for the suitable norms, this is not the case for $\d F(P) \cdot (\bm{\Pi}_P(P-P_*))$ and $\d F(P) \cdot (\bm M^{-1} R(P))$ for reasons made clear in the next section.

\begin{remark}
  In our simulations, the computation of $\d A(P) \cdot X$ for $X\in\Tc_P\Mc_\Nc$ is performed by forward-mode automatic
  differentiation using the \texttt{ForwardDiff.jl} Julia package~\cite{revelsForwardModeAutomaticDifferentiation2016}.
\end{remark}

We summarize the results of this section in \autoref{fig:forces_cs}, displaying the combination of these bounds: the successive operator norms result in very inaccurate bounds (from six to eleven orders of magnitude) for the error on the forces.

\begin{figure}[h!]
  \includegraphics[width=\linewidth]{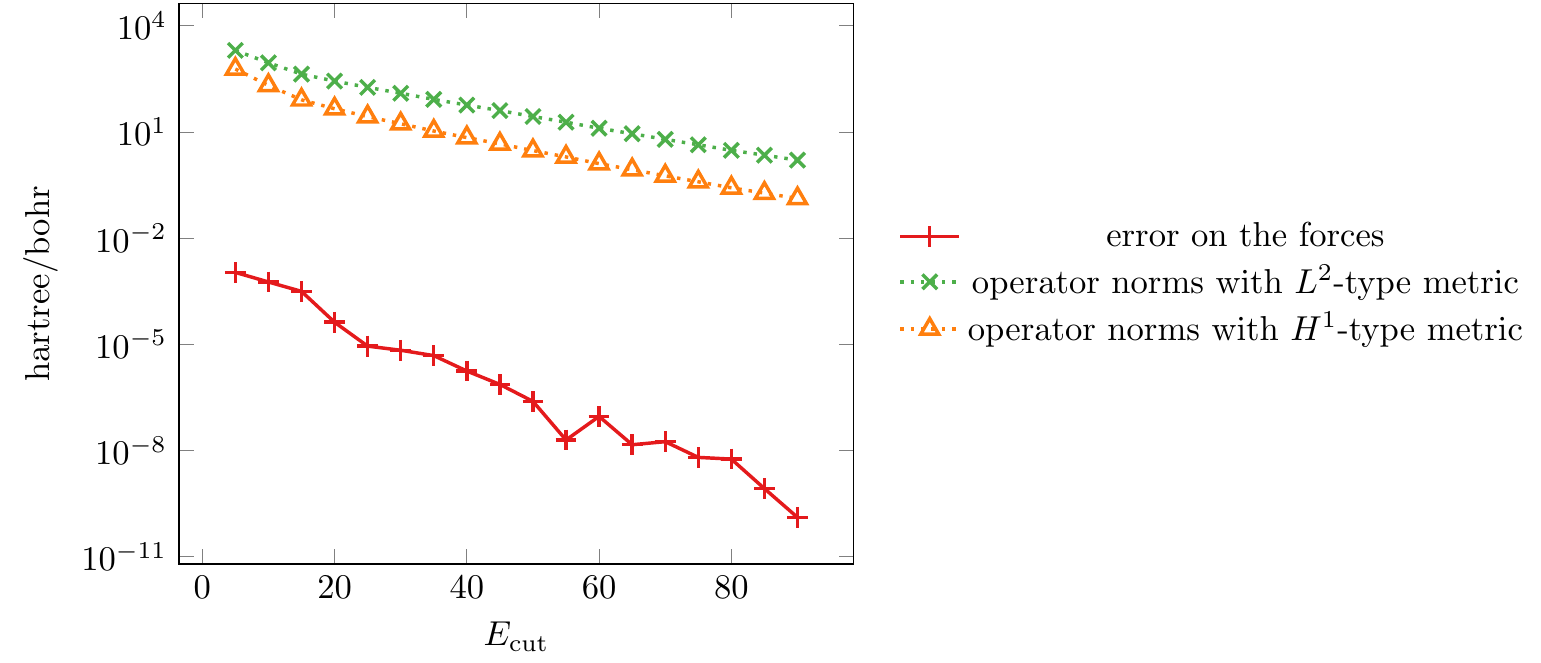}
  \caption{Combination of the error estimate \eqref{eq:qoi_bound} on the interatomic forces with
    the error estimate on the error in $L^2$-type norm \eqref{eq:badbound2} and $H^1$-type norm \eqref{eq:bound_P_M}. The inaccuracy of the bounds accumulates and results in
    extremely inaccurate bounds, from six to eleven orders of magnitude.}
  \label{fig:forces_cs}
\end{figure}

\section{Improved error bounds based on frequencies splitting}
\label{sec:sec4}

\subsection{Spectral decomposition of the error}
In the previous section, we saw that even if $\bm{\Pi}_P(P-P_*)$ and $\bm{M}^{-1}R(P)$
are asymptotically equivalent in suitable norms, replacing the former by the latter in \eqref{eq:qoi_lin} when $A=F$ (interatomic forces)
results in a large error, even in the asymptotic regime.

To analyze this issue, we use the decomposition
\begin{equation}\label{eq:decomp_X}
  \Xc_{E_{\rm cut,ref}}=\Xc_\Ecut \oplus \Xc_\Ecut^\perp.
\end{equation}
Since $\Xc_\Ecut= \Span(e_{G}, \; \frac{\abs{G}^2}2 \lq \Ecut)$ and  $\Xc_\Ecut^\perp= \Span(e_{G}, \; \Ecut < \frac{\abs{G}^2}2 \lq  E_{\rm cut,ref})$,
\eqref{eq:decomp_X} corresponds to a low \emph{vs} high frequency splitting.
Using the identification of  $\Xc_{E_{\rm cut,ref}} \equiv \Cb^\Nc$ introduced in \autoref{sec:discr}, \eqref{eq:decomp_X} boils down to decomposing $\Cb^\Nc$ as
\[
  \Cb^\Nc = \Xc \oplus \Xc^\perp \quad \mbox{with} \quad  \Xc  = \prt{\begin{array}{c} \Cb^{N_b} \\ 0_{\Cb^{\Nc-N_b}} \end{array} } \quad \mbox{and} \quad \Xc^\perp = \prt{\begin{array}{c} 0_{\Cb^{N_b}} \\ \Cb^{\Nc-N_b} \end{array} }.
\]

Let $\Phi \in \Cb^{\Nc \times \Nel}$ be such that $\Phi^*\Phi=I_\Nel$ and $P = \Phi\Phi^* \in \Mc_\Nc$.
Combining the identification $\Xc_{E_{\rm cut,ref}} \equiv \Cb^\Nc$ described above with the relation \eqref{eq:identification_tangent_plane} identifying a matrix $X$ of the tangent space $\Tc_{P}\Mc_\Nc$ with a collection $\deltaPhi=(\deltaphi_1|\cdots|\deltaphi_\Nel) \in \Cb^{\Nc \times \Nel}$ of orbital variations such that $\Phi^*\deltaPhi=0$, the decomposition \eqref{eq:decomp_X} induces a decomposition of the tangent space $\Tc_P\Mc_\Nc$ into two orthogonal subspaces $\bm \Pi_{\Ecut}\Tc_P\Mc_\Nc$ and $\bm \Pi_{\Ecut}^{\perp} \Tc_P\Mc_\Nc$ (for the Frobenius inner product):
\begin{align*}
  \bm \Pi_{\Ecut} \prt{ \sum_{i=1}^{\Nel} \ket{\phi_{i}}\bra{\deltaphi_{i}} + \ket{\deltaphi_{i}}\bra{\phi_{i}}   }
  &\coloneqq \sum_{i=1}^{\Nel} \ket{\phi_{i}}\bra{\Pi_{\Xc} \deltaphi_{i}} + \ket{\Pi_{\Xc} \deltaphi_{i}}\bra{\phi_{i}} , \\
  \bm \Pi_{\Ecut}^\perp \prt{ \sum_{i=1}^{\Nel} \ket{\phi_{i}}\bra{\deltaphi_{i}} + \ket{\deltaphi_{i}}\bra{\phi_{i}}   }
  &\coloneqq \sum_{i=1}^{\Nel} \ket{\phi_{i}}\bra{\Pi_\Xc^\perp \deltaphi_{i}} + \ket{\Pi_{\Xc}^\perp \deltaphi_{i}}\bra{\phi_{i}},
\end{align*}
where $\Pi_{\Xc}$ is the orthogonal projector on $\Xc$ (for the canonical inner product of $\Cb^\Nc$) and $\Pi_\Xc^\perp=1-\Pi_\Xc$. If $P$ solves the
minimization problem \eqref{eq:pb_Nb}, we infer from the first-order optimality conditions that the residual $R(P)$ is orthogonal to
$\bm \Pi_{\Ecut} \Tc_P\Mc_\Nc$, meaning that the vectors $r_{i}(P)$ such that
\[
  R(P) = \sum_{i=1}^{\Nel} \ket{\phi_{i}}\bra{r_{i}(P)} + \ket{r_{i}(P)}\bra{\phi_{i}}
\]
belong to $\Xc^\perp$. Note that in practice, this is not exactly true for the full Kohn--Sham model because of the numerical quadrature errors involved in the treatment of the exchange-correlation terms.

Now $P-P_{*} \approx ((\bm{\Op}(P) + \bm{K}(P))|_{\Tc_P\Mc_\Nc})^{-1} R(P)$ contains two
components: one in $\bm \Pi_{\Ecut}\Tc_P\Mc_\Nc$ and one in $\bm \Pi_{\Ecut}^{\perp}\Tc_P\Mc_\Nc$. In the high-frequency subspace $\bm \Pi_{\Ecut}^\perp \Tc_P\Mc_\Nc$, the leading term in $(\bm{\Op}(P) + \bm{K}(P))|_{\Tc_P\Mc_\Nc}$ comes
from the contribution of the Laplacian arising in the Hamiltonian $h_0$, which is well approximated by the super-operator $\bm M$. This claim is supported by Proposition~\ref{prop:err_res_cvg}, in which we prove in a simplified setting that
$((\bm{\Op}(P) + \bm{K}(P))|_{\Tc_P\Mc_\Nc})^{-1} \bm \Pi_{\Ecut}^{\perp}$ is
asymptotically equivalent to $\bm M^{-1} \bm \Pi_{\Ecut}^{\perp}$.

This is what we observe in \autoref{fig:carot_1} (central and right
panels): if $P$ is the solution to \eqref{eq:pb_Nb}, the residual $R(P)$ is supported in $\bm \Pi_{\Ecut}^\perp \Tc_P\Mc_\Nc$ (up to numerical quadrature errors). In accordance with Proposition~\ref{prop:err_res_cvg}, the
difference between the error $P-P_* \approx \bm{\Pi}_P(P-P_*)$ and the preconditioned residual ${\bm M}^{-1}R(P)$
(\autoref{fig:carot_diff_res}) is smaller in Frobenius norm than the preconditioned
residual itself. This
explains our observations in \autoref{sec:simple_error_bound} that
$\normF{P-P_{*}}$ is well approximated by $\normF{\bm M^{-1} R(P)}$. However,
this does not imply that $\d F(P) \cdot \bm{\Pi}_P(P - P_*)$ is well-approximated
by $\d F(P) \cdot ({\bm M}^{-1} R(P))$. This is because the gradients $\nabla F_{j,\alpha}(P)$ are mostly supported on
low frequencies, as illustrated in \autoref{fig:carot_1}
(left panel). Although the low-frequency contribution to the error $\bm{\Pi}_P(P - P_*)$ is of
smaller magnitude than the high-frequency contribution, its contribution to
$\d F(P) \cdot (\bm{\Pi}_P(P - P_*))$ is very significant. The fact that the low-frequency error is not captured at
all by the purely high-frequency term $\bm M^{-1} R(P)$ is responsible for the poor approximation of
the error $F(P)-F_*$ by $\d F(P) \cdot ({\bm M}^{-1} R(P))$.

\begin{figure}[ !ht]
  \includegraphics[width=0.33\linewidth]{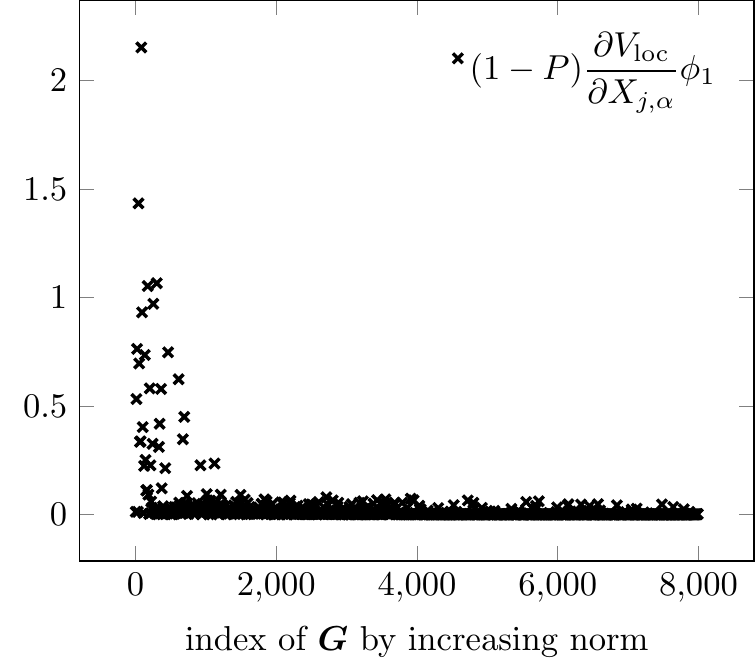}\hfill
  \includegraphics[width=0.33\linewidth]{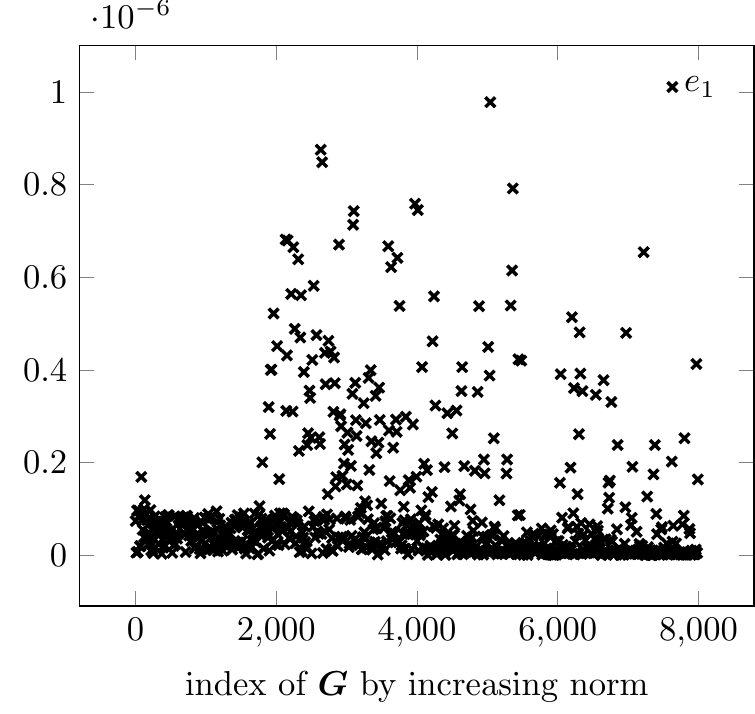}\hfill
  \includegraphics[width=0.33\linewidth]{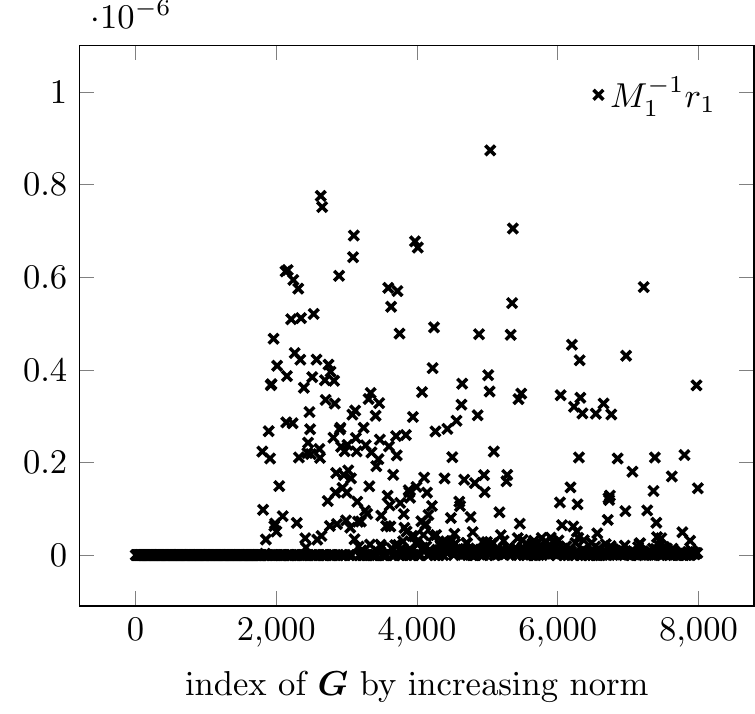}\hfill
  \caption{
    Fourier coefficients moduli in the orbital representation $\bm{\Pi}_P(P-P_*) \simeq_\Phi (e_i)_{1\lq i \lq N}$ and $\bm{M}^{-1}R(P) \simeq_\Phi (M_i^{-1}r_i)_{1\lq i \lq N}$.
    (Left) Test function $(1-P)\frac{\partial V_\text{loc}}{\partial X_{j,\alpha}}\phi_1$ (see \eqref{eq:fdir_diff}).
 }  \label{fig:carot_1}
\end{figure}
\begin{figure}[ !ht]
  \centering
  \includegraphics[width=0.33\linewidth]{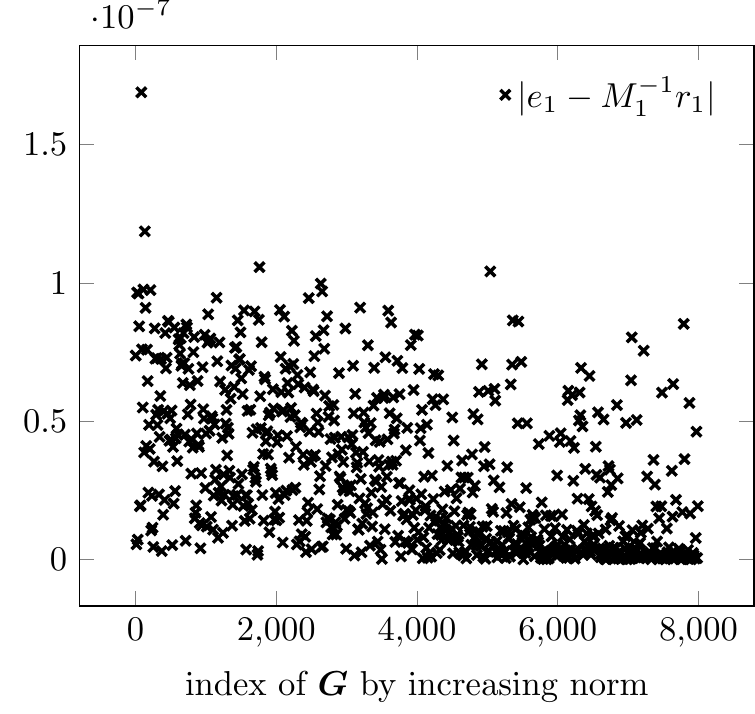}
  \caption{Fourier coefficients of the difference between the error $e_1$ and the preconditioned residual $M^{-1}_1r_1$, where $\bm{\Pi}_P(P-P_*) \simeq_\Phi (e_i)_{1\lq i \lq N}$ and $\bm{M}^{-1}R(P) \simeq_\Phi (M_i^{-1}r_i)_{1\lq i \lq N}$. Low frequencies contribute greatly.}
  \label{fig:carot_diff_res}
\end{figure}

Now that we have understood the reason why it is not possible to approximate the error $F(P) - F_*$ on the interatomic forces by the computable term $\d F(P) \cdot (\bm M^{-1} R(P))$, we propose in the next section a way to evaluate this error, based on the linearization \eqref{eq:linearize} and the frequencies splitting we just introduced.

\subsection{Improving the error estimation}

We now decompose tangent vectors and operators according to the splitting $\bm \Pi_{\Ecut}\Tc_P\Mc_\Nc$ and
$\bm \Pi_{\Ecut}^{\perp} \Tc_P\Mc_\Nc$, which we respectively label by 1 and 2 for
simplicity. In this way, the error-residual relationship can be
written in concise form with obvious notation as
\begin{align*}
  \begin{bmatrix}
    \prt{\bm \Op + \bm K}_{11} & \prt{\bm \Op + \bm K}_{12} \\
    \prt{\bm \Op + \bm K}_{21} & \prt{\bm \Op + \bm K}_{22}
  \end{bmatrix}
  \begin{bmatrix}
    P_{1} - P_{*1}\\
    P_{2}-P_{*2}
  \end{bmatrix}
  =
  \begin{bmatrix}
    R_{1}\\R_{2}
  \end{bmatrix}.
\end{align*}

Recall that $(\bm \Op(P) + \bm K(P))|_{\Tc_P\Mc_\Nc}$ is only invertible at
high cost as it has a priori nonzero values on the four components of the
operator arising from the low frequencies/high frequencies splitting of the
operator. The computational cost for the inversion is equivalent to performing
a Newton step on the reference grid.
But we can make approximations to invert
it only on the coarse grid $\Xc_\Ecut$ and approximate the low
frequency error components. In the same spirit as for the
perturbation-theory based post-processing method introduced in
\cite{cancesPostprocessingPlanewaveApproximation2020,dussonPostprocessingPlanewaveApproximation2020} and the Feshbach-Schur method analyzed in~\cite{dussonAnalysisFeshbachSchurMethod2020},
we make the following approximations:
\[
  \prt{\bm \Op + \bm K}_{21} \approx 0 \quad \mbox{and} \quad \prt{\bm \Op + \bm K}_{22} \approx \bm M_{22},
\]
which yields
\begin{align*}
  \begin{bmatrix}
    \prt{\bm \Op + \bm K}_{11} & \prt{\bm \Op + \bm K}_{12} \\
    0 & \bm M_{22}
  \end{bmatrix}
  \begin{bmatrix}
    P_{1} - P_{*1}\\
    P_{2}-P_{*2}
  \end{bmatrix}
  =
  \begin{bmatrix}
    R_{1}\\R_{2}
  \end{bmatrix}
\end{align*}
and therefore
\begin{align}
  \label{eq:schur_1}
  P_{2} - P_{*2} &\approx \bm M_{22}^{-1} R_{2},\\
  \label{eq:schur_2}
  P_{1}-P_{*1} &\approx \prt{\bm \Op + \bm K}_{11}^{-1} (R_{1} - \prt{\bm \Op + \bm K}_{12} \bm M_{22}^{-1} R_{2}).
\end{align}
This requires only a single inexpensive computation on the fine grid.
The main bottleneck is then to solve a linear system with
operator $\prt{\bm \Op + \bm K}_{11}$, which is as expensive as a
full Newton step on the coarse grid $\Xc_\Ecut$. Since $R_{1} = 0$ when
$P$ is the optimal Galerkin solution on $\Xc_\Ecut$, we can understand the
previous attempt to replace $P-P_{*}$ by $\bm M^{-1} R(P)$ as
\eqref{eq:schur_1}. Not neglecting $\prt{\bm \Op + \bm K}_{12}$
in \eqref{eq:schur_2} gives rise to a correction on the coarse space also. We denote by $R_{\rm Schur}(P)$ the new residual
\[
  R_{\rm Schur}(P) =   \begin{bmatrix}
    \prt{\bm \Op + \bm K}_{11}^{-1} (R_{1} - \prt{\bm \Op + \bm K}_{12} \bm M_{22}^{-1} R_{2}) \\ \bm M_{22}^{-1} R_{2}
  \end{bmatrix}.
\]

To illustrate the validity of these approximations, we plotted in
\autoref{fig:carot_2} the components of $r_\text{Schur}$, the orbital representation of $R_\text{Schur}$.
We see that this time, the error is well approximated by \eqref{eq:schur_2}
in the low-frequency space.

\begin{figure}[ !ht]
  \hfill
  \includegraphics[width=0.33\linewidth]{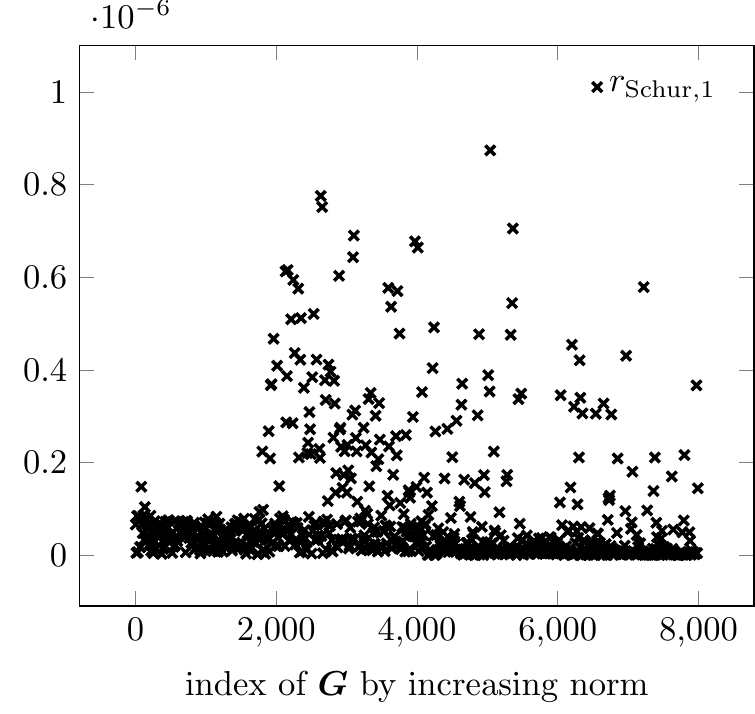}\hfill
  \includegraphics[width=0.33\linewidth]{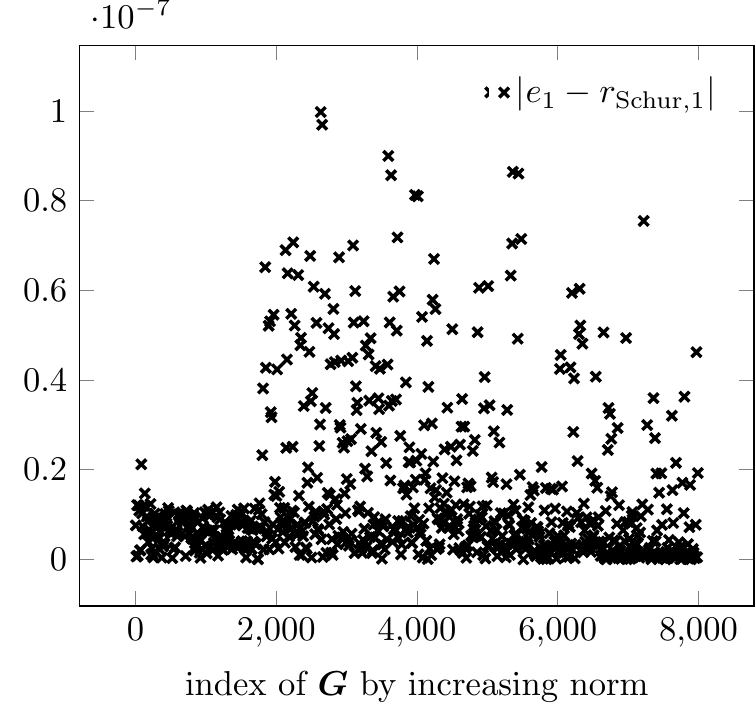}\hfill{~}
  \caption{Fourier coefficients of the new residual $r_{{\rm Schur},1}$ and its comparison to the error $e_1$, where $\bm{\Pi}_P(P-P_*) \simeq_\Phi (e_i)_{1\lq i \lq N}$ and $R_{\rm Schur}(P) \simeq_\Phi (r_{{\rm Schur},i})_{1\lq i \lq N}$.
    (Left) Components of the modified residual $r_{\text{Schur},1}$.
    (Right) Difference between the error and the new residual: low frequencies are better approximated (compare with \autoref{fig:carot_diff_res}).}
  \label{fig:carot_2}
\end{figure}

In \autoref{fig:forces_estimator}, we plot the new estimate $\d F(P) \cdot (R_\text{Schur}(P))$ of the error $F(P) - F_*$ as well as the differences
\begin{align*}
  F_\text{err} - F_* &\coloneqq F(P) - \d F(P) \cdot (\bm{\Pi}_P(P-P_*)) - F_*, \\
  F_\text{res} - F_* &\coloneqq F(P) - \d F(P) \cdot (\bm{M}^{-1}R(P)) - F_*, \\
  F_\text{Schur} - F_* &\coloneqq F(P) - \d F(P) \cdot (R_\text{Schur}(P)) - F_*,
\end{align*}
in order to have a better estimation of the improvement on the estimation of the error. With the Schur
complement method, the new estimate better matches the error than the crude one simply using the residual: the accuracy of the estimation is approximately improved by one order of magnitude.

\begin{figure}[ !ht]
  \centering
  \includegraphics[height=0.3\linewidth]{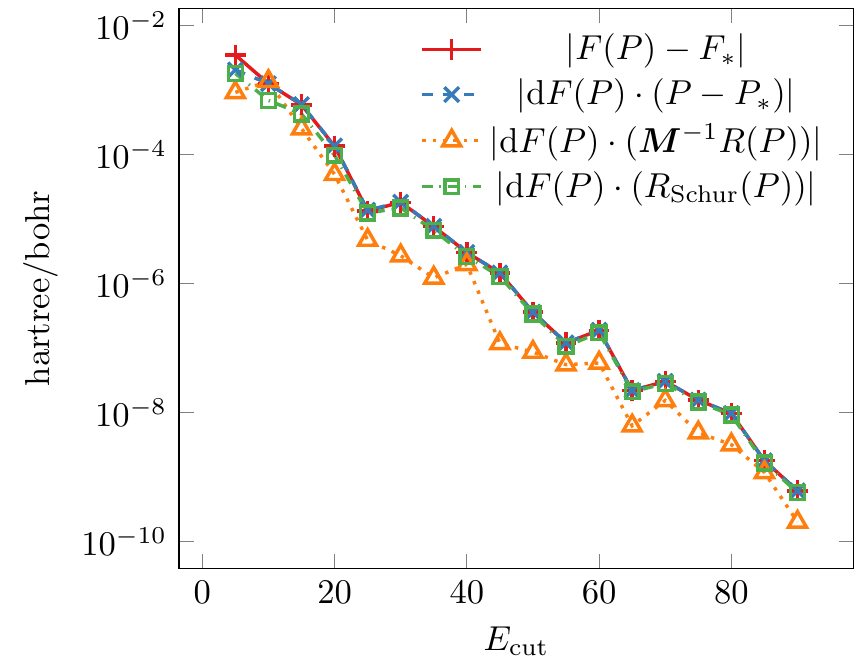}\hspace{2cm}
  \includegraphics[height=0.3\linewidth]{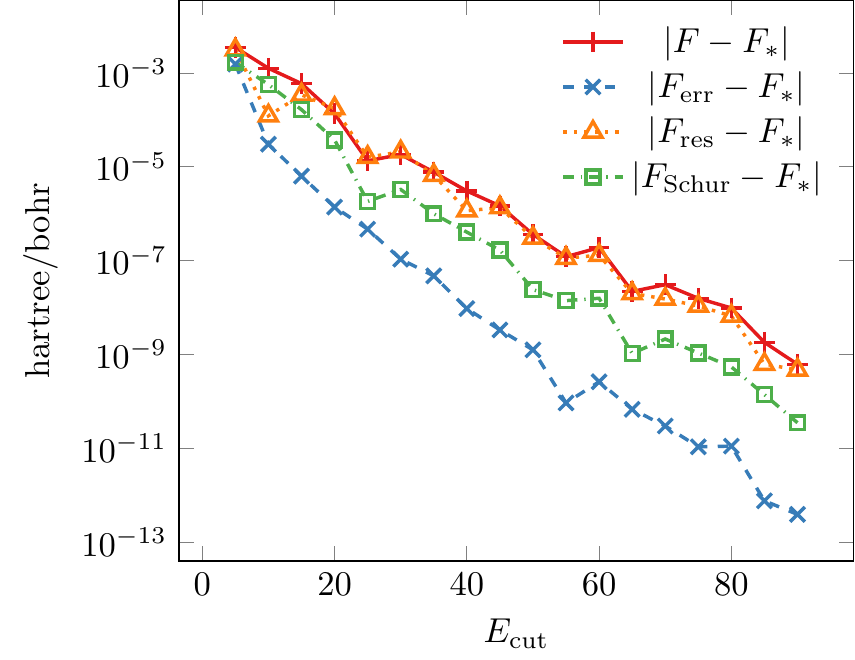}
  \caption{(Left) Estimation of the error $F(P) - F_*$ with $\d F(P)\cdot X$ where $X$ is either the exact error $\bm{\Pi}_P(P-P_*)$, the preconditioned residual $\bm{M}^{-1}R(P)$ or the modified residual $R_\text{Schur}(P)$. (Right) Enhancement of the estimation of the forces by replacing $F(P)$ with $F(P) - \d F(P)\cdot X$ where $X$ is either the exact error $\bm{\Pi}_P(P-P_*)$, the preconditioned residual $\bm{M}^{-1}R(P)$ or the modified residual $R_\text{Schur}(P)$.}
  \label{fig:forces_estimator}
\end{figure}

\begin{remark}
  The quantity $\d F(P) \cdot (R_\text{Schur}(P))$ does not yield a guaranteed
  estimator of the error on the forces as it is obtained after several
  approximations and is only valid in the asymptotic regime.
  However, it can be computed for a cost comparable to the
  one of performing a SCF step on the same grid and can be used for two main
  purposes:
  \begin{itemize}
    \item as an error bound, as the error $F(P)-F_*$ is
      reasonably well approximated by $\d F(P)\cdot(R_\text{Schur}(P))$;
    \item as a more precise approximation of the QoI, as the forces $F_j(P)$ on atom $j$ obtained by a variational approximation on a coarse grid are improved by the post-processing $F_j(P) \mapsto F_j(P) - \d F_j(P)\cdot(R_\text{Schur}(P))$.
  \end{itemize}
\end{remark}

\section{Numerical examples with more complex systems}\label{sec:experiments}

We perform the same simulations as for
silicon, but for more complex systems, namely \ch{GaAs} and \ch{TiO2}. The calculations are still performed within the LDA approximation with GTH pseudopotentials and Teter 93 exchange-correlation functional, with a $2 \times 2 \times 2$ $k$-point grid to discretize the Brillouin zone, and the reference solutions are obtained for $E_{\rm cut,ref} = 125$ Ha. We describe here the numerical setting for both systems.
\begin{description}
  \item[\ch{GaAs}] We use the usual periodic lattice for the FCC phase of \ch{GaAs}, with lattice constant $10.68$ bohrs, close to but not exactly at the equilibrium configuration in order to get nonzero forces. The \ch{Ga} atom is placed at fractional coordinates $(\frac 1 8, \frac 1 8, \frac 1 8)$ and the \ch{As} atom at fractional coordinates $(-\frac 1 8, -\frac 1 8, -\frac 1 8)$. The $\ch{Ga}$ atom is then displaced by $\frac 1 {15} (0.24, -0.33, 0.12)$ to get nonzero forces. In this setting, the reference values for the energy is $E_* = -8.572 $ Ha and the interatomic forces are, in hartree/bohr,
    \[
    F_* = \begin{bmatrix}
    -0.0448 &   0.0448 \\
    0.0722  & -0.0722 \\
    -0.0251 &  0.0251 \\
    \end{bmatrix},
    \]
    where the first column are the forces acting on the $\ch{Ga}$ atom in each direction, and the second column are the forces acting on the $\ch{As}$ atom.

  \item[\ch{TiO2}] We use the MP-2657 configuration in the primitive cell from the Materials Project~\cite{perssonMaterialsDataTiO22014}. We apply the small displacement $\frac 1 5 (0.22, -0.28, 0.35)$ to the equilibrium position of the first \ch{Ti} atom to get nonzero forces. In this setting, the reference values for the energy is $E_* = -71.589 $ Ha and the interatomic forces are, in hartree/bohr,
    \[\footnotesize
    F_* = \begin{bmatrix}
 -2.88 &  0.641 &  3.80 &  0.753 & -1.57 & -0.745 \\
  3.10 & -0.919 & -3.09 & -1.45 &  0.800 &  1.56 \\
  0.136 &  0.403 & -0.368 & -0.786 &  0.251 &  0.364 \\
    \end{bmatrix},
    \]
    where the first two columns are the forces acting on the two $\ch{Ti}$ atoms in each direction, and the other columns are the forces acting on the four $\ch{O}$ atoms.
\end{description}

\medskip

We plot in \autoref{fig:linearization_all} the energy, density and forces obtained after a Newton step on the fine grid starting from the variational solution on the coarse grid given by $\Ecut$, for \ch{GaAs} and \ch{TiO2}. The fast establishment of the asymptotic regime is confirmed for the two new systems as, even for small $\Ecut$'s, the so-obtained QoIs are orders of magnitude more accurate than the ones obtained by the variational solution on the coarse grid.

\begin{figure}[h!]
  \centering
  \begin{tabular}{ccc}
    & \ch{GaAs} & \\
    \includegraphics[width=0.30\linewidth]{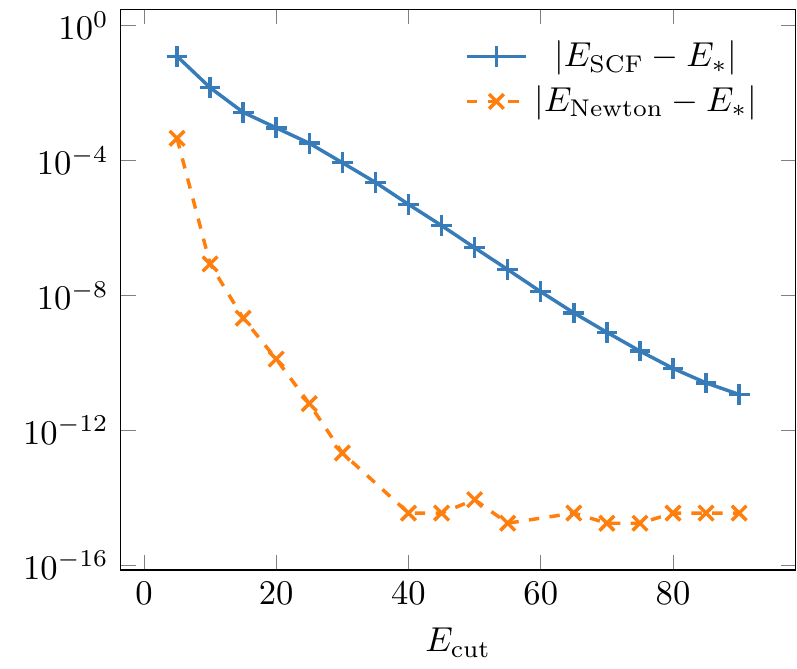}
    &
    \includegraphics[width=0.30\linewidth]{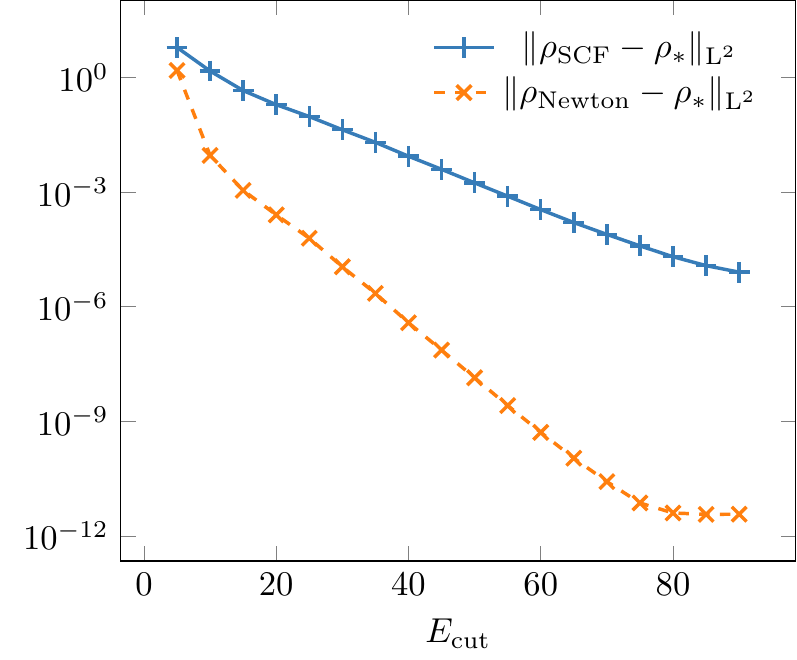}
    &
    \includegraphics[width=0.30\linewidth]{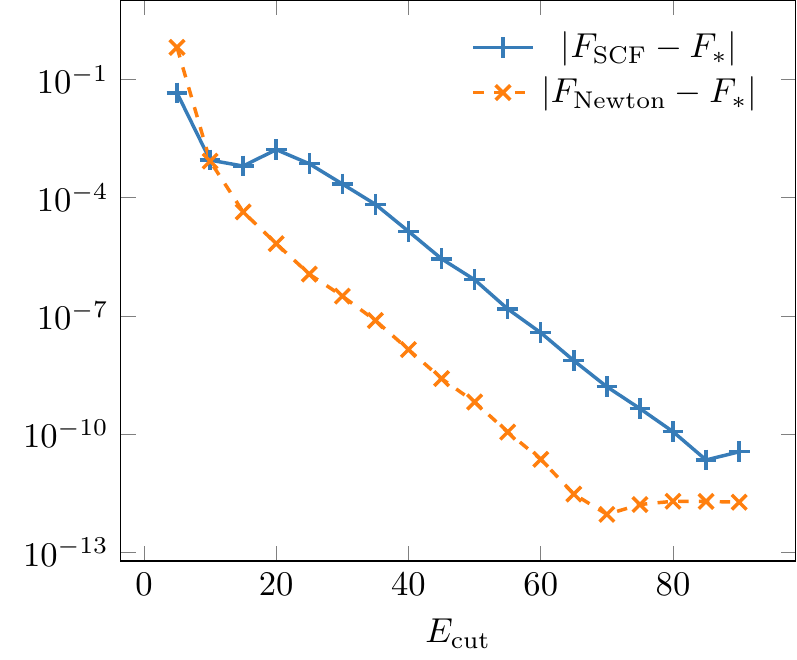} \\ \\
    & \ch{TiO2} & \\
    \includegraphics[width=0.30\linewidth]{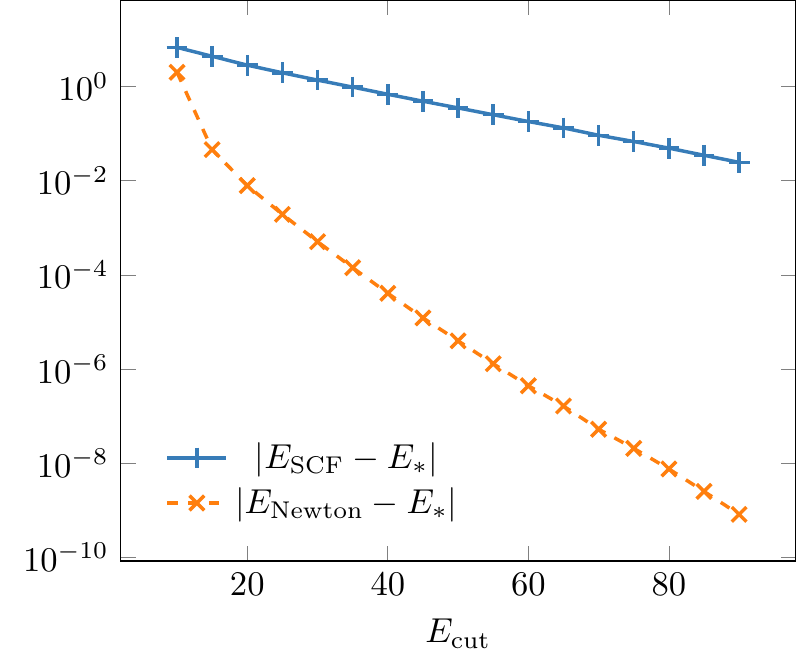}
    &
    \includegraphics[width=0.30\linewidth]{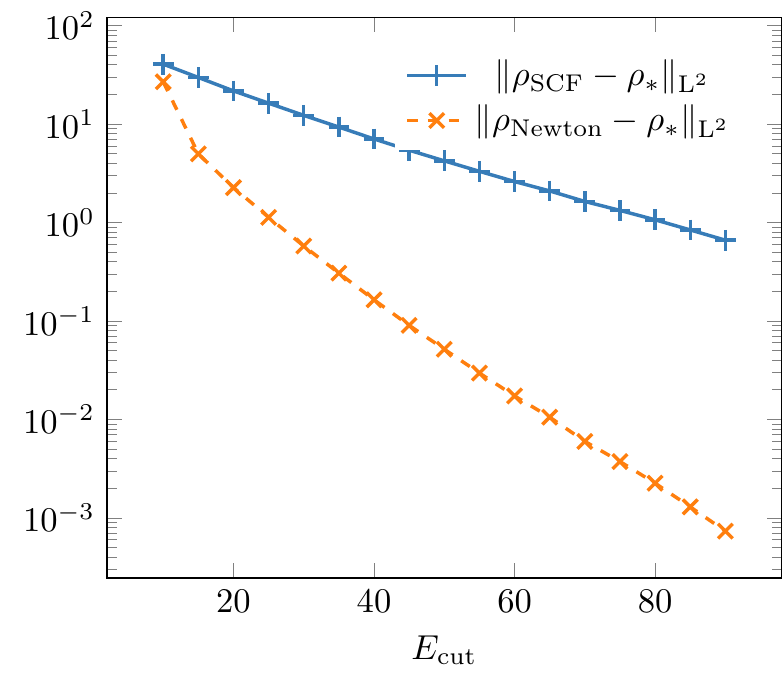}
    &
    \includegraphics[width=0.30\linewidth]{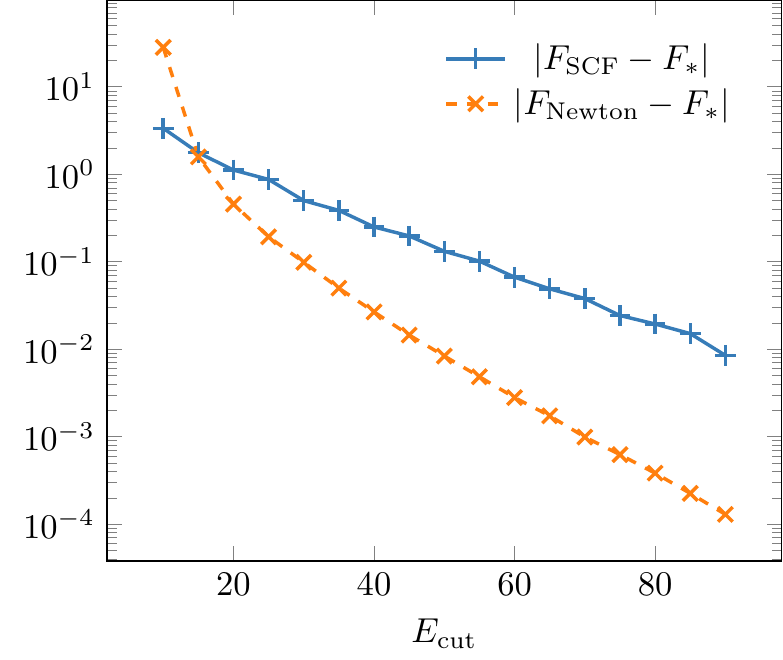} \\
  \end{tabular}
  \caption{
    Errors of some QoI as functions of
    $\Ecut$ (reference solution is obtained with $E_{\rm cut,ref}=125$ Ha) for \ch{GaAS} and \ch{TiO2}.
    Solid lines: errors obtained with the variational solution in the space $\Xc_{\Ecut}$.
    Dashed lines: errors obtained with one Newton step on the reference grid, starting from the variational solution in the space $\Xc_{\Ecut}$.
    Left panel: energy (hartree), central panel: discrete $\L^2$ norm of the density, right panel: interatomic forces (hartree/bohr). To be compared with \autoref{fig:linearization}.}
  \label{fig:linearization_all}
\end{figure}

\medskip

We plot in \autoref{fig:forces_estimator_all} the estimation of the actual error ${F(P) - F_*}$ with $\d F(P)\cdot X$ where $X$ is either $\bm{\Pi}_P(P-P_*)$, $R(P)$ or $R_{\rm Schur}(P)$.
In \autoref{fig:forces_differences_all}, we plot the improvement of the estimation of the forces $F(P) - \d F(P)\cdot X$ where $X$ is either $\bm{\Pi}_P(P-P_*)$, $R(P)$ or $R_{\rm Schur}(P)$. Just as for silicon, the estimation is well improved with the modified residual $R_{\rm Schur}$. Note that in the \ch{GaAs} case, there is a plateau for high $\Ecut$'s. This phenomenon is explained in the remark below.

\begin{figure}[h!]
  \centering
  \begin{tabular}{cc}
    \ch{GaAs} & \ch{TiO2} \\ \\
    \includegraphics[height=0.345\linewidth]{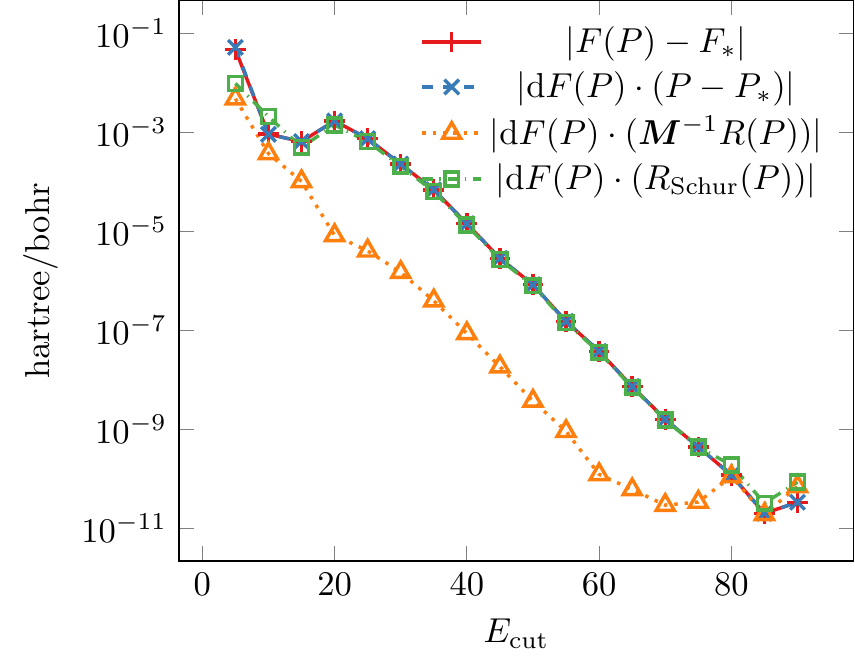} &
    \includegraphics[height=0.35\linewidth]{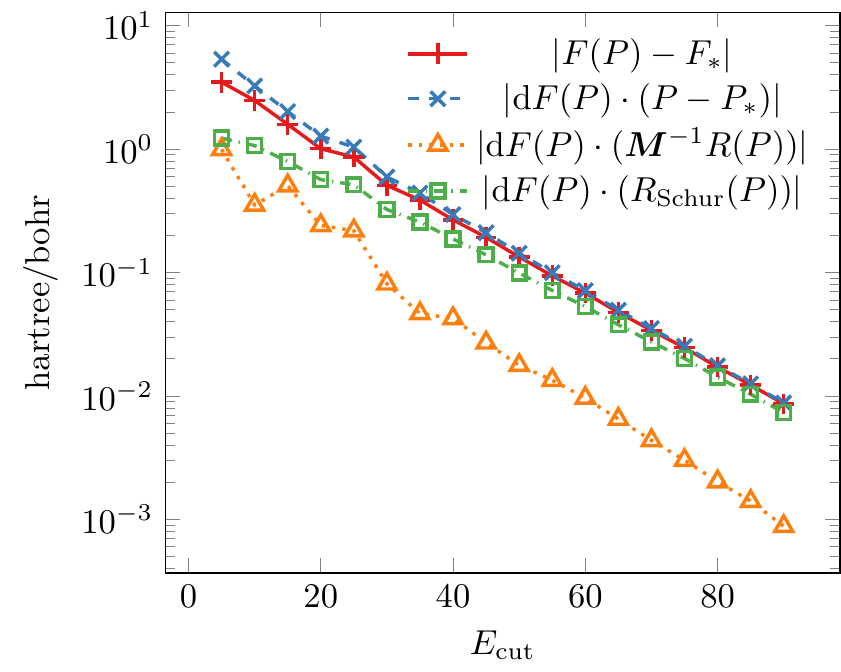}
  \end{tabular}
  \caption{Estimation of the error $F(P) - F_*$ with $\d F(P)\cdot X$ where $X$ is either the exact error $\bm{\Pi}_P(P-P_*)$, the preconditioned residual $\bm{M}^{-1}R(P)$ or the modified residual $R_\text{Schur}(P)$.  To be compared with \autoref{fig:forces_estimator} (Left).}
  \label{fig:forces_estimator_all}
\end{figure}

\begin{figure}[h!]
  \centering
  \begin{tabular}{cc}
    \ch{GaAs} & \ch{TiO2} \\ \\
    \includegraphics[height=0.34\linewidth]{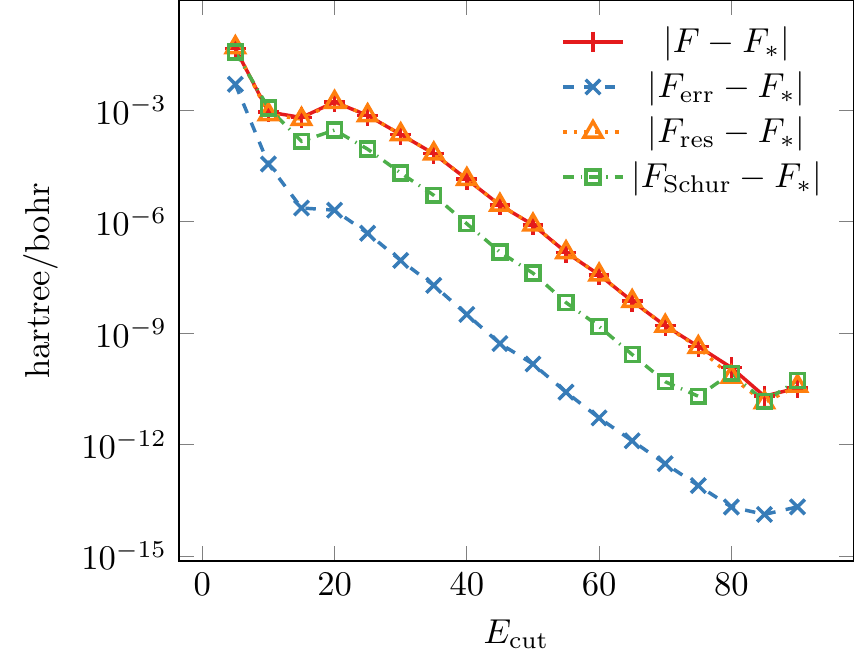} &
    \includegraphics[height=0.35\linewidth]{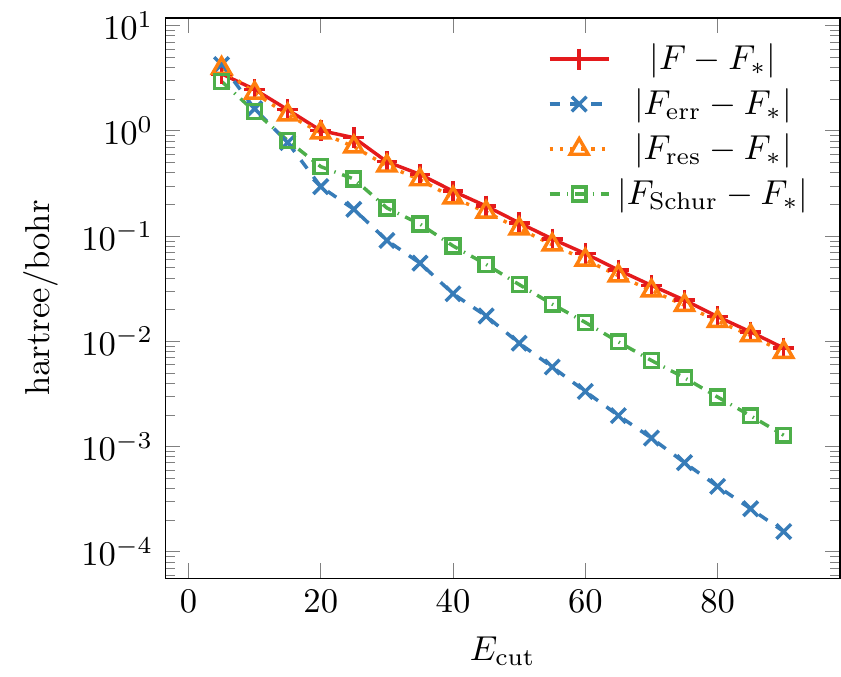}
  \end{tabular}
  \caption{Enhancement of the estimation of the forces by replacing $F(P)$ with $F(P) - \d F(P)\cdot X$ where $X$ is either the exact error $\bm{\Pi}_P(P-P_*)$, the preconditioned residual $\bm{M}^{-1}R(P)$ or the modified residual $R_\text{Schur}(P)$. To be compared with \autoref{fig:forces_estimator} (Right).}
  \label{fig:forces_differences_all}
\end{figure}

\begin{remark}\label{rmk:sampl}
  The plateau observed \autoref{fig:forces_estimator_all} and
  \autoref{fig:forces_differences_all} for \ch{GaAs}
  comes from the numerical quadrature scheme used to compute the
  exchange-correlation energy and the corresponding matrix elements.
  In fact, we also observed such plateaus for silicon and \ch{TiO2}
  with the default quadrature scheme of DFTK, but these disappeared by
  using 8 times as many numerical quadrature points. With this more
  accurate numerical quadrature scheme, the plateau for \ch{GaAs} is
  lower but still visible. It disappears when further increasing the
  number of quadrature points, at the price of longer computations.
\end{remark}

\section{Conclusion}

In this work, we have investigated methods to estimate the error on
interatomic forces resulting from plane-wave discretizations of the
Kohn--Sham equations. On the systems we investigated, we find the following:
\begin{itemize}
\item Linearizing the equations around a solution is a good
  approximation, even for energy cut-offs as small as 5 hartree (\autoref{fig:linearization}).
\item The naive approach based on the computation of operator norms
  proves to be extremely inefficient, overestimating the error by
  several orders of magnitude. This is the case even when using
  appropriate $H^{1}$-type norms (\autoref{fig:forces_cs}). The reason
  is that the discretization error is mostly made up of high frequency
  components, whereas quantities of interest are mostly supported on
  low frequencies, resulting in very suboptimal Cauchy-Schwarz
  inequalities (\autoref{fig:carot_1}).
\item Replacing directly the error by the preconditioned residual
  yields reasonable estimates of the errors, but they are not
  systematic upper bounds (\autoref{fig:forces}).
\item A Schur approach based on a low/high frequency splitting
  systematically improves the solution and gives reliable estimates of
  the error (\autoref{fig:forces_estimator}), at the price of more computational work.
\end{itemize}

Our results validate on realistic test cases and for properties such
as interatomic forces the frequency splitting approach already
introduced in
\cite{cancesPostprocessingPlanewaveApproximation2020,dussonPostprocessingPlanewaveApproximation2020,dussonAnalysisFeshbachSchurMethod2020}.Thanks
to the modular nature of DFTK and the use of automatic
differentiation, the implementation of our estimates is relatively simple and
convenient. It is publicly available at
\url{https://github.com/gkemlin/paper-forces-estimator}. The algorithm
proceeds in two steps: i) the computation of the residual on the fine
grid and ii) a linear system solve involving the Jacobian on the coarse grid. The
computational cost of step i) is negligible compared to that of step
ii), which is roughly that of a full self-consistent computation on
the coarse grid. Therefore, for roughly twice the cost of a standard
computation, one obtains an accurate approximation of the
discretization error on the interatomic forces (or, equivalently, a better estimate of the
latter).

The scope of this work is limited to gapped systems at zero temperature and to the study
of the discretization error. Interesting perspectives for future work
include the application of this methodology to the error resulting
from an incomplete self-consistent cycle, and to finite-temperature
models, including metals (see \cite{herbst2022robust} for an extension
of the linearized equations to the finite-temperature case).
\section*{Acknowledgements}
The authors would like to thank Michael F. Herbst for fruitful
discussions and help with DFTK, and Niklas Schmitz and Markus Towara for the
implementation of forward differentiation into DFTK. This project has
received funding from the European Research Council (ERC) under the
European Union's Horizon 2020 research and innovation programme (grant
agreement No 810367).
G.D. acknowledges the support of the Ecole des Ponts-ParisTech, as this
work was carried out in the framework of an associated researcher
position at CERMICS.

\appendix

\section{Mathematical justification}

\label{sec:proof}

The purpose of this appendix is to explain mathematically in a
simplified setting the observation in \autoref{sec:simple_error_bound} that
$\normF{\bm M^{-1/2} \bm{\Pi}_{P} R(P)}$ was a good approximation of
$\normF{\bm{M}^{1/2}\bm{\Pi}_{P}(P-P_*)}$. For this purpose, we work
in a slightly different framework than the one we used in the rest of
the paper, and consider the infinite-dimensional version of
Problem~\eqref{eq:pb} associated with the periodic Gross--Pitaevskii
model in dimension $d \lq 3$, which reads as
\begin{equation}\label{eq:GP_DM}
  E_*\coloneqq\min\{E(P), \; P \in \Mc_\infty\},
\end{equation}
with  $\Mc_\infty\coloneqq\{P \in \Sc(\L^2_\#) \, | \, P^2=P, \; \Tr(P)=1, \; \Ran(P) \subset \H^1_\# \}$ and
  $E(P) \coloneqq \Tr((-\Delta +V) P) + \frac 12 \int_\Gamma \rho_P^2$.
Here $\Sc(\L^2_\#)$ denotes the space of self-adjoint operators on $\L^2_\#$, $V$ a given function of $\L^\infty_\#$, and $\rho_P$ the density of $P$. The condition $\Ran(P) \subset \H^1_\#$ ensures that both the linear and nonlinear terms in the energy functional $E(P)$ are well-defined and finite.
It is convenient to rewrite \eqref{eq:GP_DM} in the orbital framework. Any state $P \in \Mc_\infty$ is rank-1 and such that $\Ran(P) \subset \H^1_\#$. It can therefore be represented by a function $\phi \in \H^1_\#$ such that $\norm{\phi}_{\L^2_\#}=1$ through the relation $P=\ket{\phi}\bra{\phi}$ (using Dirac's notation). The orbital formulation of problem~\eqref{eq:GP_DM} reads
\begin{equation}\label{eq:GP_MO}
  E_*\coloneqq\min\{{\mathcal E}^{\rm GP}(\phi), \; \phi \in \H^1_\#, \; \norm{\phi}_{\L^2_\#}=1\},
  \end{equation}
with ${\mathcal E}^{\rm GP}(\phi)\coloneqq \int_\Gamma |\nabla\phi|^2 + \int_\Gamma V|\phi|^2 + \frac 12 \int_\Gamma |\phi|^4$.
It is well-known (see e.g. the Appendix of~\cite{cancesNumericalAnalysisNonlinear2010}) that the minimizer of~\eqref{eq:GP_DM} is unique, and that the set of solutions of~\eqref{eq:GP_MO} is $(e^{i\alpha}\phi_*)_{\alpha \in \Rb}$, where $(\lambda_*,\phi_*) \in \Rb \times \H^1_\#$ is the unique solution to
\begin{equation}\label{eq:GP_study}
  \begin{cases}
    -\Delta \phi_* + V\phi_* + \phi_*^3 = \lambda_* \phi_*,\\
    \norm{\phi_*}_{\L^2_\#} = 1, \quad \phi_* > 0 \mbox{ on } \Rb^d.
  \end{cases}
\end{equation}
We consider the variational approximation of~\eqref{eq:GP_DM} in the finite dimensional space
\[
  \Xc_N\coloneqq\Span(e_{G}, \; |{G}|^2/2 \lq N)
\]
corresponding to a plane-wave discretization with energy cut-off $\Ecut=N$. We denote by $\Pi_N$ the $\L^2_\#$-orthogonal projector on $\Xc_N$ and by $\Pi_N^\perp\coloneqq1-\Pi_N$. For $N$ large enough, the approximate ground-state $P_N$ is unique and can be represented by a unique function $\phi_N$ real-valued and positive on $\Rb^3$ (see~\cite{cancesNumericalAnalysisNonlinear2010}), and it holds
\[
  \begin{cases}
    -\Delta \phi_N+\Pi_N\prt{ V\phi_N - \phi_N^3 } = \lambda_N \phi_N,\\
    \norm{\phi_N}_{\L^2_\#} = 1,
  \end{cases}
\]
for some uniquely defined $\lambda_N \in \Rb$. In addition, we have $\phi_* \in \H^2_\#$ and
\begin{equation}\label{eq:CV_GP}
  \norm{\phi_N -\phi_*}_{\H^2_\#} \mathop{\longrightarrow}_{N \to \infty} 0 \quad \mbox{and} \quad | \lambda_N -\lambda_* | \mathop{\longrightarrow}_{N \to \infty} 0.
\end{equation}

Using similar notation as the one used in the rest of the paper, we introduce the following quantities:
\begin{itemize}
  \item $\Pi_{\phi_N}^\perp$ is the orthogonal projector (for the $\L^2_\#$ inner product) onto $\phi_N^\perp$;
  \item $A_N$ is the self-adjoint operator on $\phi_N^\perp$ defined by
    \begin{align} \label{eq:def_AN}
      A_N \coloneqq (\Op_{N} + K_{N})
    \end{align}
    where $\Op_N$ and $K_N$ represent, in the orbital framework, the super-operators $\bm{\Op}(P_N)|_{T_{P_N}\Mc_\infty}$ and $\bm{K}(P_N)|_{T_{P_N}\Mc_\infty}$. We have
    \begin{align}
      \label{eq:OmegaEcut}
      \Forall \psi_N \in \phi_N^\perp, \quad \Op_N \psi_N &= \Pi_{\phi_N}^\perp\prt{-\Delta + V + \phi_N^2 - \lambda_N}\psi_N,\\
      \label{eq:KEcut}
      \Forall \psi_N \in \phi_N^\perp, \quad K_N \psi_N &= \Pi_{\phi_N}^\perp\prt{2\phi_N^2\psi_N};
    \end{align}
  \item $M_N^{1/2}$ is the restriction of the operator $\Pi_{\phi_N}^\perp (1-\Delta)^{1/2} \Pi_{\phi_N}^\perp$ to the invariant subspace $\phi_N^\perp$.
\end{itemize}

We then have the following result, which justifies in this case the claim made in \autoref{sec:simple_error_bound} that $\bm{M}^{-1/2}\bm{M}^{1/2}((\bm{\Op}(P) + \bm{K}(P))|_{\Tc_P\Mc_\Nc})^{-1}\bm{M}^{1/2}$ is close to identity on the subspace of high-frequency Fourier modes. It also justifies that $\bm{\Pi}_P(P-P_*) \approx \bm{\bm{M}}^{-1} R(P)$ as $\bm{M}^{-1/2}$ is a uniformly bounded operator and $\bm{M}^{-1/2} R(P)$ is high-frequency:
\begin{align*}
  \bm{\Pi}_P(P-P_*) &\approx \bm{M}^{-1/2}\bm{M}^{1/2}((\bm{\Op}(P) + \bm{K}(P))|_{\Tc_P\Mc_\Nc})^{-1}\bm{M}^{1/2} \bm{M}^{-1/2}R(P) \\
  & \approx \bm{M}^{-1/2} \bm{M}^{-1/2} R(P) = \bm{M}^{-1}R(P).
\end{align*}

\begin{proposition}
  \label{prop:err_res_cvg}
  We have
  \[
    \lim_{N\to\infty}
    \norm{M_{N}^{1/2}(\Op_{N} + K_{N})^{-1}M_{N}^{1/2}
      - I_{\Xc_{N}^\perp}}_{\Xc_{N}^\perp\to \L^2_\#} = 0.
  \]
\end{proposition}

\begin{proof} Let $W_N \coloneqq V + 3\phi_N^2 - \lambda_N - 1$ and $W_*\coloneqq V+3\phi_*^2-\lambda_*-1$. In view of \eqref{eq:CV_GP}, $W_N$ converges to $W_*$ in $\L^\infty_\#$ when $N$ goes to infinity.
  It also follows from~\cite[Lemma~1]{cancesNumericalAnalysisNonlinear2010} that the self-adjoint operator
  \begin{align*}
    \widetilde A_* &\coloneqq-\Delta + V + 3\phi_*^2 - \lambda_*=(1-\Delta)+W_* \\
    &=(1-\Delta)^{1/2} \Big(1 + (1-\Delta)^{-1/2} W_{*} (1-\Delta)^{-1/2}\Big)(1-\Delta)^{1/2}
  \end{align*}
  is coercive, hence, by the Lax--Milgram lemma, defines a continuous isomorphism from $\H^1_\#$ to $\H^{-1}_\#$. We denote by $\widetilde A_*^{-1}$ its inverse, seen as a bounded operator from $\H^{-1}_\#$ to $\H^1_\#$, so that $B_*\coloneqq(1-\Delta)^{1/2}\widetilde A_*^{-1}(1-\Delta)^{1/2}$ defines a bounded operator on $\L^2_\#$.

  Using the convergence results \eqref{eq:CV_GP} and standard
  perturbation theory it follows that for $N$ large enough, the operator $B_N\coloneqq(1-\Delta)^{1/2}\widetilde A_N^{-1}(1-\Delta)^{1/2}$, where $\widetilde A_N\coloneqq(1-\Delta)+W_N$, is bounded on $\L^2_\#$ uniformly in $N$, and that we have
  \begin{equation}\label{eq:BN}
    B_N= \prt{ 1 +(1-\Delta)^{-1/2}  W_N (1-\Delta)^{-1/2} }^{-1} = 1 -
    B_N (1-\Delta)^{-1/2} W_N (1-\Delta)^{-1/2} .
  \end{equation}

  We now compute the action of the operator $M_{N}^{1/2}A_N^{-1}M_{N}^{1/2} : \phi_{N}^{\perp}
  \to \phi_{N}^{\perp}$, relating it to $\widetilde A_{N}^{-1}$ and $B_{N}$, with $A_N$ defined in \eqref{eq:def_AN}.
  Let $\deltaphi_N \in X_N^\perp$. As $\phi_N \in X_N$, we have $X_N^\perp \subset \phi_N^\perp$ so that $\deltaphi_N \in X_N^\perp \subset \phi_N^\perp$, and $M_{N}^{1/2}\deltaphi_N=(1-\Delta)^{1/2}\deltaphi_N \in X_N^\perp \subset \phi_N^\perp$, where we used that $X_N$ and $X_N^\perp$ are invariant subspaces of the operator $(1-\Delta)^{1/2}$.
  Let $v_N\coloneqq A_N^{-1}M_{N}^{1/2}\deltaphi_N=A_N^{-1}(1-\Delta)^{1/2}\deltaphi_N \in \phi_N^\perp$. Using~\eqref{eq:OmegaEcut} and~\eqref{eq:KEcut}, we get
  \[
    \Pi_{\phi_N}^\perp \prt{-\Delta + V + 3\phi_N^2 - \lambda_N}v_N=(1-\Delta)^{1/2}\deltaphi_N, \quad \mbox{\ie} \quad
    \Pi_{\phi_N}^\perp \widetilde A_Nv_N = (1-\Delta)^{1/2}\deltaphi_N,
  \]
  and therefore,
  \[
    \widetilde A_Nv_N=(1-\Delta)^{1/2}\deltaphi_N+\alpha_N\phi_N,
  \]
  where $\alpha_N = - \frac{\cro{\phi_N , \widetilde A_N^{-1} (1-\Delta)^{1/2}\deltaphi_N}_{\L^2_\#}}{\cro{\phi_N , \widetilde A_N^{-1} \phi_N}_{\L^2_\#}}= - \frac{\cro{(1-\Delta)^{-1/2}\phi_N , B_N\deltaphi_N}_{\L^2_\#}}{\cro{\phi_N , \widetilde A_N^{-1} \phi_N}_{\L^2_\#}} \in {\mathbb R}$ is characterized by the constraint $v_N \in \phi_N^\perp$.
  We thus obtain
  \[
    v_N = \widetilde A_N^{-1} (1-\Delta)^{1/2}\deltaphi_N - \frac{\cro{(1-\Delta)^{-1/2}\phi_N , B_N\deltaphi_N}_{\L^2_\#}}{\cro{\phi_N , \widetilde A_N^{-1} \phi_N}_{\L^2_\#}}\widetilde A_N^{-1}\phi_N,
  \]
  and therefore, as $v_N \in \phi_N^\perp$,
  {\footnotesize\begin{align*}
    & M_{N}^{1/2}A_N^{-1}M_{N}^{1/2} \deltaphi_{N} - \deltaphi_N \\
    &\quad =M_N^{1/2}v_N-\deltaphi_N \\
    &\quad = \Pi_{\phi_N}^\perp (1-\Delta)^{1/2}v_N-\deltaphi_N \\
    &\quad =(1-\Delta)^{1/2}v_N-\cro{\phi_N,(1-\Delta)^{1/2}v_N}_{\L^2_\#} \phi_N - \deltaphi_N \\
    &\quad =(B_N-1)\deltaphi_N- \cro{\phi_N,B_N\deltaphi_N}_{\L^2_\#}  \phi_N \\&\quad\phantom{=}- \frac{\cro{(1-\Delta)^{-1/2}\phi_N , B_N\deltaphi_N}_{\L^2_\#}}{\cro{\phi_N , \widetilde A_N^{-1} \phi_N}_{\L^2_\#}}
    \prt{ (1-\Delta)^{1/2}\widetilde A_N^{-1}\phi_N - \cro{\phi_N, (1-\Delta)^{1/2} \widetilde A_N^{-1}\phi_N}_{\L^2_\#} \phi_N } \\
    &\quad =(B_N-1)\deltaphi_N- \cro{\phi_N,B_N\deltaphi_N}_{\L^2_\#}  \phi_N \\&\quad\phantom{=}- \frac{\cro{(1-\Delta)^{-1/2}\phi_N , B_N\deltaphi_N}_{\L^2_\#}}{\cro{\phi_N , \widetilde A_N^{-1} \phi_N}_{\L^2_\#}}
    \prt{ B_N(1-\Delta)^{-1/2}\phi_N - \cro{\phi_N, B_N(1-\Delta)^{-1/2} \phi_N}_{\L^2_\#} \phi_N } \\
    &\quad  = (B_N-1)\deltaphi_N- \cro{\phi_N,(B_N-1)\deltaphi_N}_{\L^2_\#}  \phi_N \\& \quad \phantom{=}- \frac{\cro{(1-\Delta)^{-1/2}\phi_N , (B_N-1)\deltaphi_N}_{\L^2_\#}}{\cro{\phi_N , \widetilde A_N^{-1} \phi_N}_{\L^2_\#}}
    \!\! \prt{ B_N(1-\Delta)^{-1/2}\phi_N - \cro{\phi_N, B_N(1-\Delta)^{-1/2} \! \phi_N}_{\L^2_\#} \!\! \phi_N },
  \end{align*}}
  where we used the fact that $\deltaphi_N \in \Xc_N^\perp$, while $\phi_N$ and $(1-\Delta)^{-1/2}\phi_N$ belong to $\Xc_N$. Using again \eqref{eq:CV_GP} we obtain that for $N$ large enough,
  \begin{align*}
    \norm{M_{N}^{1/2}A_{N}^{-1}M_{N}^{1/2}
      - I_{\Xc_{N}^\perp}}_{\Xc_{N}^\perp\to \L^2_\#} &= \sup_{\deltaphi_N \in \Xc_N^\perp}  \frac{\norm{M_{N}^{1/2}A_N^{-1}M_{N}^{1/2} \deltaphi_{N}-\deltaphi_N}_{\L^2_\#}}{\norm{\deltaphi_N}_{L^2_\#}} \\
      & \lq C_* \norm{(B_N-1)\Pi_N^\perp}_{\L^2_\# \to \L^2_\#},
  \end{align*}
  where
  \[
    C_*\coloneqq3 + \frac{\norm{\phi_*}_{\H^{-1}_\#}}{\cro{\phi_*,\widetilde A_*^{-1}\phi_*}_{\L^2_\#}}
    \times 2\norm{B_*}_{\L^2_\#\to\L^2_\#}\norm{\phi_*}_{\H_\#^{-1}}.
  \]
  Finally, using \eqref{eq:BN}, we have
  \begin{align*}
    \norm{(B_N-1)\Pi_N^\perp}_{\L^2_\# \to \L^2_\#} &\lq \norm{B_N
    (1-\Delta)^{-1/2}
    W_N}_{\L^2_\# \to \L^2_\#} \norm{(1-\Delta)^{-1/2} \Pi_N^\perp}_{\L^2_\# \to \L^2_\#} \\
    & \lq  \norm{B_N}_{\L^2_\# \to \L^2_\#} \norm{W_N}_{\L^\infty_\#} (1+2N)^{-1/2} \mathop{\longrightarrow}_{N \to 0} 0,
  \end{align*}
  since  $\norm{B_N}_{\L^2_\# \to \L^2_\#}$ and
  $\norm{W_N}_{\L^\infty_\#}$ are uniformly bounded in $N$. This concludes
  the proof.
\end{proof}

\bibliography{./biblio}
\bibliographystyle{abbrv}
\end{document}